\documentclass[12pt,letterpaper]{amsart}
\usepackage[none]{hyphenat}
\usepackage[utf8]{inputenc}
\usepackage[T1]{fontenc}
\usepackage{amsfonts}
\usepackage[leqno]{amsmath}
\usepackage{amssymb, comment, float}
\usepackage[dvipsnames]{xcolor}
\usepackage[foot]{amsaddr}
\usepackage[shortlabels]{enumitem}
\usepackage{mathdots}
\usepackage{todonotes}
\usepackage[footskip=0.5in, headheight = 0.5in, top=1.25in, bottom=1.25in, right=1in, left=1in]{geometry}
\usepackage{hyperref}
\hypersetup{
colorlinks,
linkcolor=blue,
urlcolor=black,
citecolor=blue,}
\usepackage[center]{caption}
\usepackage{eufrak}
\usepackage{marvosym}
\usepackage{latexsym}
\usepackage{tikz}
\usepackage{subfig}
\usepackage{thm-restate}

\newtheorem{theorem}{Theorem}[section]
\newtheorem{lemma}[theorem]{Lemma}

\newtheorem{question}[theorem]{Question}

\DeclareMathOperator{\tw}{tw}

\DeclareMathOperator{\cl}{cl}
\DeclareMathOperator{\lv}{lv}

\def\dd{\hbox{-}}

\newcommand{\mf}{\mathfrak}
\newcommand{\mca}{\mathcal}
\newcommand{\poi}{\mathbb N}

\newcommand{\blackref}[1]{\hypersetup{linkcolor=black}\ref{#1}\hypersetup{linkcolor=blue}}

\newcounter{tbox}
\newcommand{\sta}[1]{\medskip\medskip\refstepcounter{tbox}\noindent{\parbox{\textwidth}{(\thetbox) \textit{#1}}}\vspace*{0.3cm}}
\newcommand{\mylongtitle}[1]{%
	\ifodd\value{page}%
		\protect\parbox{0.97\linewidth}{#1}\hfill%
	\else%
		\hfill\protect\parbox{0.97\linewidth}{#1}%
	\fi%
}

\title[Induced Cycles of Many Lengths]{Induced Cycles of Many Lengths}

\author{Maria Chudnovsky$^{\dagger}$}
\author{Ilya Maier$^{\ddagger}$}

\thanks{$^{\dagger}$ Princeton University, Princeton, NJ, USA. Supported by NSF Grant DMS-2348219, NSF Grant CCF-2505100, AFOSR grant FA9550-25-1-0275 and a Guggenheim Fellowship}
\thanks{$^{\ddagger}$ Princeton University, Princeton, NJ, USA. Supported by a scholarship of the German Academic Exchange Service (DAAD)}
\date{\today}

\begin{document}
\maketitle

\begin{abstract}

	Let $G$ be a graph and let $\cl(G)$ be the number of distinct induced cycle lengths in $G$.
	We show that for $c,t\in \poi$, every graph $G$ that does not contain an induced subgraph isomorphic to $K_{t+1}$ or $K_{t,t}$ and satisfies $\cl(G) \le c$ has bounded treewidth.
	As a consequence, we obtain a polynomial-time algorithm for deciding whether a graph $G$ contains induced cycles of at least three distinct lengths.

\end{abstract}

\section{Introduction}

The set of all positive integers is denoted by $\poi$ and the set of all positive integers no greater than $k$ is denoted by $\poi_k$.
For a set $X$, we denote the set of all subsets of $X$ by $2^X$.
Graphs in this paper have finite vertex sets, no loops, and no parallel edges.
Let $G$ and $H$ be two graphs.
$G$ is said to \textit{contain} $H$ if $G$ contains an induced subgraph isomorphic to $H$.
$G$ is \textit{$H$-free} if $G$ does not contain $H$.
A \textit{path in} $G$ is an induced subgraph of $G$ that is a path.
The \textit{length} of a path is the number of edges in it.
We allow paths consisting of a single vertex, which have length zero.

For $X \subseteq V(G)$, we denote by $G[X]$ the subgraph of $G$ induced by $X$, and we refer to $X$ and $G[X]$ interchangeably.
For $X\subseteq V(G)$, we define \textit{the neighborhood} of $X$ in $G$, denoted $N_G(X)$, to be the set of all vertices in $V(G) \setminus X$ with at least one neighbor in $X$.
When the graph $G$ is clear, we simply use $N(X)$.
If $X=\{x\}$, then we write $N(x)$ for $N(\{x\})$.
For standard graph theoretic terminology, the reader is referred to \cite{Diestel}.

For a graph $G=(V(G), E(G))$, a \textit{tree decomposition} $(T, \chi)$ of $G$ consists of a tree $T$ and a map $\chi: V(T) \rightarrow 2^{V(G)}$ with the following properties:
\begin{enumerate}[\rm (i)]
	\item For every $v \in V(G)$, there exists $t \in V(T)$ such that $v \in \chi(t)$.
	\item For every $uv\in E(G)$, there exists $t \in V(T)$ such that $u,v \in \chi(t)$.
	\item For every $v \in V(G)$, the subgraph of $T$ induced by $\{t \in V(T) \mid v \in \chi(t)\}$ is connected.
\end{enumerate}

For every $t \in V(T)$, we refer to $t$ as a \textit{node of} and to $\chi(t)$ as a \textit{bag of} $(T, \chi)$.
The \textit{width} of a tree decomposition $(T, \chi)$ is $\max_{t \in V(T)}|\chi(t)|-1$.
The \textit{treewidth} of $G$, denoted by $\operatorname{tw}(G)$, is the minimum width of a tree decomposition of $G$.
Additionally, we define $\cl(G)$ to be the number of distinct induced cycle lengths in $G$, i.e., \[\cl(G) = \left|\{l\in \poi \mid G \text{ has an induced cycle of length } l\} \right|.\]

Berger, Seymour, and Spirkl \cite{NonShortest} provided the first polynomial-time algorithm for deciding whether there is a non-shortest $u\dd v$ path in $G$ for $u,v\in V(G)$ and Chiu and Lu \cite{non-shortest-faster} later provided a faster algorithm for this problem.
This algorithm can be straightforwardly applied to test whether all induced cycles in $G$ have the same length.
However, there is no simple way to apply this algorithm to test whether a graph $G$ also contains induced cycles of at least three distinct lengths.
In this paper, we provide a polynomial-time algorithm for doing exactly that.

\begin{restatable}{theorem}{algoThree}
\label{thm:algo_3}
	
	Let $G$ be a graph.
	There exists an algorithm with running time $O(|V(G)|^{7.75})$ that decides whether $\cl(G)\ge 3$ and outputs a list of three induced cycles of distinct lengths if they exist.

\end{restatable}

As a generalized version and a core subroutine of this algorithm, we also show that we can decide if $\cl(G)\ge c$ for every $c\in \poi$ in linear time given that $G$ has bounded treewidth.

\begin{restatable}{theorem}{algoBounded}
\label{thm:algo_bounded}

	Let $c,k\in \poi$ be constants and let $G$ be a graph with $\tw(G)\le k$.
	There exists an algorithm with running time $O(|V(G)|)$ that decides whether $\cl(G)\ge c$ and outputs a list of $c$ induced cycles of distinct lengths if they exist.

\end{restatable}

The majority of this paper is devoted to the proof of the following structural result that identifies a necessary condition on the structure of graphs with large treewidth.

\begin{restatable}{theorem}{mainResult}
\label{thm:main_result}

	For all $c,t \in \poi$, there is a constant $f_{\blackref{thm:main_result}} = f_{\blackref{thm:main_result}}(c,t) \in \poi$ such that for every graph $G$ with $\tw(G) \ge f_{\blackref{thm:main_result}}$ one of the following holds.
	\begin{enumerate}[\rm (a)]
		\item $G$ contains $K_{t+1}$ or $K_{t,t}$.
		\item $\cl(G) \ge c$.
	\end{enumerate}

\end{restatable}

We leave the following question open.

\begin{question}
	
	For all fixed $k\in \poi$, does there exist a polynomial-time algorithm to decide whether a graph $G$ contains at least $k$ induced cycles of distinct lengths?
	
\end{question}

\subsection{Key structures}

We start by defining some of the key structures used throughout this paper to explain the outline of the paper.
We will restate these definitions in the sections where they are first used.

Let $G$ be a graph and let $X$ and $Y$ be two disjoint subsets of $V(G)$.
We say that $X$ and $Y$ are \textit{anticomplete} (or that $X$ {\em is anticomplete to} $Y$) if there is no edge in $G$ with an end in $X$ and an end in $Y$.
We say that {\em $X$ has a neighbor in $Y$} if $X$ is not anticomplete to $Y$.
For $x \in V(G)$, we say that $x$ is \textit{anticomplete to $Y$} if $\{x\}$ and $Y$ are anticomplete.
For a set $\mca{X}$ of subsets of $V(G)$, we write $V(\mca{X})=\bigcup_{X \in \mca{X}} X$.

Let $k \in \poi$ and let $P$ be a $k$-vertex path in $G$.
Then we write $P=p_1\dd \cdots \dd p_k$ to mean that $V(P)=\left\{p_1, \ldots, p_k\right\}$ and $E(P)=\left\{p_i p_{i+1}: i \in \poi_{k-1}\right\}$.
For vertices $u, v \in V(P)$, we denote by $u \dd P \dd v$ the subpath of $P$ from $u$ to $v$.
We call the vertices $p_1$ and $p_k$ the \textit{ends} of $P$ and refer to $P \backslash\left\{p_1, p_k\right\}$ as the \textit{interior} of $P$, denoted $P^*$.
For a set of paths $\mca P$, we define the \textit{interior of} $\mca P$, denoted $\mca P^*$, to be the set of interiors of paths in $\mca P$, i.e., $\mca P^* = \{ P^* \mid P \in \mca P\}$.
 
A {\em subdivision} of $G$ is a graph $H$ obtained from $G$ by repeatedly adding a vertex $v$ and replacing an edge $uw$ with a path $u\dd v\dd w$ such that $v$ is only adjacent to $u$ and $w$.
A {\em proper subdivision} of $G$ is a subdivision of $G$ where every edge has been subdivided at least once.

\newcommand{\starDefinition}{
	A \textit{(subdivided) star in $G$} is an induced subgraph of $G$ isomorphic to (a subdivision of) the complete bipartite graph $K_{1, s}$ for some $s \geq 3$.
	Let $S$ be a (subdivided) star in $G$.
	The {\em root} of $S$ is the unique vertex of $S$ of degree at least three.
	The {\em leaf set} of $S$ is the set of vertices of $S$ of degree one.
}

\newcommand{\domeDefinition}{
	An {\em $(s,d)$-dome in $G$} is a pair $(S,P)$ such that
	\begin{itemize}
		\item $S$ is a subdivided star in $G$ with leaf set $L$ of size $s$;
		\item $P$ is a path in $G$ disjoint from $S$;
		\item $V(S)\setminus L$ and $V(P)$ are anticomplete in $G$;
		\item every vertex $v\in L$ has a neighbor in $P$; and
		\item every vertex $u\in P$ is adjacent to at most $d$ vertices in $L$ (see Figure \ref{fig:dome}).
	\end{itemize}
	The {\em root} of the dome $(S,P)$ is the root of $S$.
	A dome is {\em aligned} if there exists an ordering $l_1,\dots,l_s$ of the vertices of $L$ and $s$ pairwise disjoint subpaths $u_i \dd P \dd v_i$ of $P$ for $i\in \poi_s$ (possibly of length zero) such that
	\begin{itemize}
		\item there is an end $u$ of $P$ such that $P$ traverses $u,u_1,v_1,\dots,u_s,v_s$ in this order; and
		\item for every $i\in \poi_s$, all neighbors of $l_i$ in $P$ are contained in $u_i \dd P \dd v_i$.
	\end{itemize}
	Note that for every $d\in \poi$, an aligned $(s,d)$-dome is also an aligned $(s,1)$-dome.}
	
\newcommand{\bananaDefinition}{
	For $q\in \poi$, we define a \textit{$q$-banana in $G$ with ends $x,y$}, denoted $B_{x,y}$, to be a pair $(\{x,y\},\mca P)$ where $\mca P = \{P_1,\dots,P_q\}$ is a set of $q$ pairwise internally vertex-disjoint $x\dd y$ paths in $G$.
	We define a \textit{$q$-theta in $G$ with ends $x,y$}, denoted $\Theta_{x,y}$, to be a $q$-banana $(\{x,y\},\{P_1,\dots,P_q\})$ in $G$ such that for all distinct $i,j \in \poi_q$, the paths $P_i\setminus \{x,y\}$ and $P_j\setminus \{x,y\}$ are anticomplete to each other (see Figure \ref{fig:banana}).
	Let $B_{x,y} = (\{x,y\},\mca P)$ be a $q$-banana in $G$ with $\mca P = \{P_1,\dots,P_q\}$.
	We say that $B_{x,y}$ is \textit{rigid} if for all distinct $i,j \in \poi_q$, the paths $P_i\setminus \{x,y\}$ and $P_j\setminus \{x,y\}$ are not anticomplete to each other (see Figure \ref{fig:banana}).
	We say that $B_{x,y}$ is \textit{$l$-dense} for $l\in \poi$ if there is no $l$-theta with ends $x,y$ in $G[V(\mca P)]$.
	We define $V(B_{x,y})$ to be the set $V(\mca P)$ and refer to $B_{x,y}$ and $G[V(B_{x,y})]$ interchangeably (note that this also applies to a theta $\Theta_{x,y}$ in $G$).
	For convenience, we define $B_{x,y}^*$ to be the set $\mca P^*$.
}

\newcommand{\tailedThetaDefinition}{
	For a triple of distinct vertices $x,y,z$, we define an \textit{$(x,z)$-tailed $q$-theta in $G$}, denoted $\Theta_{x,y,z}$, to be a pair $(\Theta_{x,y}, P)$ such that
	\begin{itemize}
		\item $\Theta_{x,y}$ is a $q$-theta in $G$,
		\item $P$ is a $y \dd z$ path in $G$, called a \textit{tail}, and
		\item $P\setminus \{y\}$ and $V(\Theta_{x,y})\setminus\{y\}$ are disjoint and anticomplete.
	\end{itemize}
	We define $V(\Theta_{x,y,z})$ to be the set $V(\Theta_{x,y})\cup V(P)$ and refer to $\Theta_{x,y,z}$ and $G[V(\Theta_{x,y,z})]$ interchangeably.
}

\newcommand{\kiteDefinition}{
	For $k,q\in\poi$, we define an \textit{$(x,z)$-tailed $(k,q)$-kite $\mca K$ in $G$} to be a set of $k$ $(x,z)$-tailed $q$-thetas $\allowbreak \{\Theta_{x,y_1,z},\dots, \Theta_{x,y_k,z}\}$ in $G$ such that for all distinct $i,j \in \poi_k$, $V(\Theta_{x,y_i,z}) \cap V(\Theta_{x,y_j,z}) = \{x,z\}$ (see Figure \ref{fig:kite}).
	We define $V(\mca K)$ to be the set $\cup_{i\in \poi_k} V(\Theta_{x,y_i,z})$ and refer to $\mca K$ and $G[V(\mca K)]$ interchangeably.
	For $i \in \poi_k$, let $\Theta_{x,y_i,z} = (\Theta_{x,y_i},P_i)$.
	We call $\mca K$ \textit{clean} if for every $i \in \poi_k$, every vertex of $V(\Theta_{x,y_i})\setminus \{x,y_i\}$ has degree two in $\mca K$, and the set $\{y_1,\dots,y_k\}$ is a stable set (see Figure \ref{fig:kite}).
}

\starDefinition

\domeDefinition

\begin{figure}
	\centering
	\includegraphics[width=0.35\linewidth]{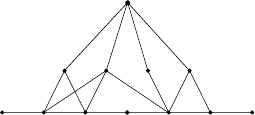}
	\caption{A $(4,3)$-dome that is not aligned.}
	\label{fig:dome}
\end{figure}

\bananaDefinition

\begin{figure}
	\centering
	\includegraphics[width=0.6\linewidth]{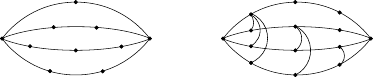}
	\caption{A 4-theta (left) and a rigid 4-banana (right).}
	\label{fig:banana}
\end{figure}

\tailedThetaDefinition

\kiteDefinition

\begin{figure}
	\centering
	\includegraphics[width=0.7\linewidth]{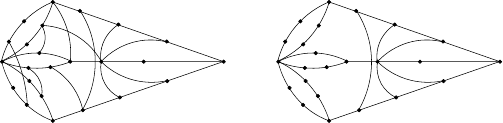}
	\caption{A $(3,2)$-kite (left) and a clean $(3,2)$-kite (right).}
	\label{fig:kite}
\end{figure}

\subsection{Proof outline and organization}

We start by outlining the proof of Theorem \ref{thm:main_result}.
In Section \ref{sec:obstructions}, we show that if a graph $G$ contains one of the basic obstructions (see Section \ref{sec:obstructions} for a precise statement) or a large dome, then $\cl(G)\ge c$.
In Section \ref{sec:reducing_to_a_banana}, we apply this result to reduce Theorem \ref{thm:main_result} to Theorem \ref{thm:banana_has_cycles}, which states that a graph containing a large dense banana contains a large dome.
We then prove Theorem \ref{thm:banana_has_cycles} in Sections \ref{sec:analyzing_bananas}, \ref{sec:banana_to_kite} and \ref{sec:general_step}.

In Section \ref{sec:analyzing_bananas}, we analyze the structure of rigid bananas.
We show how to obtain either a theta $\Theta_{x,y}$ where every neighbor of $x$ in $V(\Theta_{x,y})$ has degree two in $V(\Theta_{x,y})$, or a large dome.
In Section \ref{sec:banana_to_kite}, we apply this result to obtain many tailed thetas, which together form a kite.
Then we show how to either make this kite clean or obtain a large dome.
In Section \ref{sec:general_step}, we recursively obtain a sequence of clean kites, which we use to obtain a large dome $(S,P)$.
This completes the proof of Theorem \ref{thm:banana_has_cycles}, and thus, also Theorem \ref{thm:main_result}.

Then, in Section \ref{sec:algo}, we describe our ``key subroutine'', Theorem \ref{thm:algo_bounded}.
The input of this algorithm is a graph $G$ with bounded treewidth.
The algorithm finds a tree decomposition $(T, \chi)$ of $G$ and, using dynamic programming over $T$, either finds $c$ induced cycles of distinct lengths in $G$, or determines that $\cl(G) < c$.

Finally, in Section \ref{sec:three_cycles}, we prove Theorem \ref{thm:algo_3}.
Suppose that $G$ contains $K_3$ or $K_{2,2}$.
Then we can build on the algorithm that finds a non-shortest $u\dd v$ path from \cite{non-shortest-faster} to decide whether $\cl(G) \ge 3$.
Thus, we may assume that $G$ is $\{K_3,K_{2,2}\}$-free.
In this case, we apply Theorem \ref{thm:main_result} to reduce the problem to the case where $G$ has bounded treewidth and then apply Theorem \ref{thm:algo_bounded}.
This proves Theorem \ref{thm:algo_3}.

\section{Some good cases}
\label{sec:obstructions}

In this section, we establish some sufficient conditions for $\cl(G)$ to be large.

\subsection{Removing basic obstructions}

\begin{figure}
	\centering
	\includegraphics[width=0.9\linewidth]{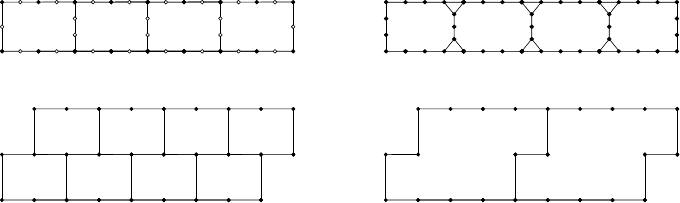}
	\caption{A proper subdivision of $W_{1\times 4}$ (top left), its line graph (top right), the graph $W_{2\times 4}$ (bottom left), and a proper subdivision of $W_{1\times 2}$ (bottom right).
Note that $W_{2\times 4}$ contains a proper subdivision of $W_{1\times 2}$.}
	\label{fig:walls}
\end{figure}

Let $G$ be a graph.
An \textit{induced minor} of $G$ is a graph obtained from an induced subgraph of $G$ by repeatedly contracting edges, deleting all arising loops and parallel edges.
For a graph $H$, we say that $G$ \textit{contains} or \textit{has} an induced $H$-minor if there is an induced minor of $G$ isomorphic to $H$.
The \textit{line graph} of $G$, denoted $L(G)$, is the graph with vertex set $E(G)$ such that $e, f \in E(G)$ are adjacent in $L(G)$ if and only if $e$ and $f$ share an end in $G$.

For all $n,m \in \poi$, we denote by $W_{n\times m}$ the $n$-by-$m$ hexagonal grid, also known as the $n$-by-$m$ \textit{wall} (Figure \ref{fig:walls}); see \cite{cliqueOrWallInducedMinor} for a precise definition ($W_{n\times m}$ is the $(n+1)$-by-$(m+1)$ elementary wall in Section 3 of \cite{cliqueOrWallInducedMinor}).
As the first step toward the proof of Theorem \ref{thm:main_result}, we prove that certain graphs with large treewidth contain induced cycles of many different lengths.
In this subsection, we prove Lemmas \ref{lem:wall_has_cycles} and \ref{lem:Ktt_minor_has_cycles} that show that for all $c,t \in \poi$ there exists $r \in \poi$ such that if a graph $G$ is $K_{t,t}$-free and contains
\begin{itemize}
	\item a subdivision of $W_{r \times r}$,
	\item the line graph of a subdivision of $W_{r \times r}$,
	\item a proper subdivision of $K_r$, or
	\item an induced $K_{r,r}$-minor,
\end{itemize}
then $\cl(G)\ge c$.
We start by introducing some definitions and Theorem \ref{thm:Ktt_minor_has_obstructions} from \cite{tw16}.

A \textit{constellation} is a graph $\mathfrak{c}$ in which there is a stable set $S_{\mathfrak{c}}$ such that every component of $\mathfrak{c} \backslash S_{\mathfrak{c}}$ is a path, and every vertex $s \in S_{\mathfrak{c}}$ has at least one neighbor in every component of $\mathfrak{c} \backslash S_{\mathfrak{c}}$.
We say that $\mathfrak{c}$ is \textit{ample} if no two vertices in $S_{\mathfrak{c}}$ have a common neighbor.
We denote by $\mca{L}_{\mathfrak{c}}$ the set of all components of $\mathfrak{c} \backslash S_{\mathfrak{c}}$ (each of which is a path), and denote the constellation $\mathfrak{c}$ by the pair $(S_{\mathfrak{c}}, \mca{L}_{\mathfrak{c}})$.
For $l,s \in \poi$, by an $(s, l)$-\textit{constellation} we mean a constellation $\mathfrak{c}$ with $\left|S_{\mathfrak{c}}\right|=s$ and $\left|\mca{L}_{\mathfrak{c}}\right|=l$ (see Figure \ref{fig:constellation}).

\begin{figure}
	\centering
	\includegraphics[width=0.35\linewidth]{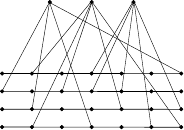}
	\caption{An ample $(3,4)$-constellation.}
	\label{fig:constellation}
\end{figure}

The following theorem is immediate from Theorem 2.3 in \cite{tw16}.

\begin{theorem}[Chudnovsky, Hajebi, Spirkl \cite{tw16}]
\label{thm:Ktt_minor_has_obstructions}

	For all $l, r, s \in \poi$, there is a constant $f_{\blackref{thm:Ktt_minor_has_obstructions}}=f_{\blackref{thm:Ktt_minor_has_obstructions}}\left(l, r, s \right) \in \poi$ such that every graph $G$ with an induced $K_{f_{\blackref{thm:Ktt_minor_has_obstructions}}, f_{\blackref{thm:Ktt_minor_has_obstructions}}}$-minor contains one of the following:
	\begin{enumerate}[\rm (a)]
		\item $K_{r,r}$, a subdivision of $W_{r \times r}$, or the line graph of a subdivision of $W_{r \times r}$;
		\item an ample $(s, l)$-constellation.
	\end{enumerate}

\end{theorem}

We will also need the following well-known result.

\begin{theorem}[Erd\H{o}s, Szekeres \cite{ErdosSzekeres}]
\label{thm:erdos_szekeres}

	For every $r \in \poi$, every sequence of real numbers with length at least $r^2+1$ contains a nondecreasing or a nonincreasing subsequence of length $r+1$.

\end{theorem}

Now, we are ready to prove Lemmas \ref{lem:wall_has_cycles} and \ref{lem:Ktt_minor_has_cycles}.

\begin{lemma}
\label{lem:wall_has_cycles}
	
	For every $c\in \poi$, there is a constant $f_{\blackref{lem:wall_has_cycles}}=f_{\blackref{lem:wall_has_cycles}}(c)$ such that if a graph $G$ contains
	\begin{enumerate}[\rm (a)]
		\item a subdivision of $W_{f_{\blackref{lem:wall_has_cycles}} \times f_{\blackref{lem:wall_has_cycles}}}$,
		\item the line graph of a subdivision of $W_{f_{\blackref{lem:wall_has_cycles}} \times f_{\blackref{lem:wall_has_cycles}}}$, or
		\item a proper subdivision of $K_{f_{\blackref{lem:wall_has_cycles}}}$,
	\end{enumerate}
	then $\cl(G) \ge c$.
	
\end{lemma}

\begin{proof}

	Define $f_{\blackref{lem:wall_has_cycles}} = 4c^2+2$.
	Suppose that $G$ contains an induced subgraph $H$ isomorphic to a subdivision of $W_{f_{\blackref{lem:wall_has_cycles}}\times f_{\blackref{lem:wall_has_cycles}}}$ or a proper subdivision of $K_{f_{\blackref{lem:wall_has_cycles}}}$.
	Then observe that $H$ contains an induced subgraph $H'$ isomorphic to a proper subdivision of $W_{1\times c^2}$ (see Figure \ref{fig:walls}).
	Let $H'$ be formed by the paths $P,Q$ and $R_1,\dots,R_{c^2+1}$ where each $R_i$ is a path from $P$ to $Q$ (view $H'$ as a subdivided ladder with rails $P$ and $Q$, and let $R_1,\dots,R_{c^2+1}$ be its rungs, indexed consecutively along the rails).

	By Theorem \ref{thm:erdos_szekeres}, there is a subsequence $R_{i_1},\dots,R_{i_{c+1}}$ such that the lengths of the paths in this subsequence are either nondecreasing or nonincreasing.
	We may assume that the lengths are nondecreasing by reindexing the paths.
	Then the induced subgraph $H''$ formed by the paths $P,Q$ and $R_{i_1},\dots,R_{i_{c+1}}$ contains at least $c$ induced cycles of strictly increasing lengths.
	Thus, $\cl(G)\ge c$ as required.

	Now, suppose that $G$ contains an induced subgraph $H$ isomorphic to the line graph of a subdivision of $W_{f_{\blackref{lem:wall_has_cycles}}\times f_{\blackref{lem:wall_has_cycles}}}$.
	Observe that $H$ contains an induced subgraph $H'$ isomorphic to the line graph of a proper subdivision of $W_{1\times c^2}$.
	By an argument analogous to the one in the previous paragraph, it follows that $\cl(G)\ge c$ as required.

\end{proof}

\begin{lemma}
\label{lem:Ktt_minor_has_cycles}

	For all $c,t\in \poi$, there is a constant $f_{\blackref{lem:Ktt_minor_has_cycles}} = f_{\blackref{lem:Ktt_minor_has_cycles}}(c,t)$ such that if a $K_{t,t}$-free graph $G$ has an induced $K_{f_{\blackref{lem:Ktt_minor_has_cycles}},f_{\blackref{lem:Ktt_minor_has_cycles}}}$-minor, then $\cl(G) \ge c$.

\end{lemma}

\begin{proof}

	Define \begin{align*}
		r&= \max \{t,f_{\blackref{lem:wall_has_cycles}}(c)\}\\
		l &= (c+1)(c!+1)\\
		f_{\blackref{lem:Ktt_minor_has_cycles}} &= f_{\blackref{thm:Ktt_minor_has_obstructions}}(l,r,c+1)
	\end{align*}
	By Theorem \ref{thm:Ktt_minor_has_obstructions}, $G$ contains
	\begin{enumerate}[\rm (a)]
		\item \label{point:Ktt_minor_has_cycles_a} $K_{r,r}$, a subdivision of $W_{r \times r}$, the line graph of a subdivision of $W_{r \times r}$, or
		\item an ample $(c+1, l)$-constellation.
	\end{enumerate}
	
	Suppose that $\cl(G) < c$.
	Since $G$ is $K_{t,t}$-free and $t \le r$, Lemma \ref{lem:wall_has_cycles} implies that $G$ contains an ample $(c+1, l)$-constellation $\mf c$.
	Let $\mca L_{\mf c} = \{L_1,\dots,L_l\}$.
	For every $i\in \poi_l$, let $L_i'=v_i^1\dd \dots \dd v_i^{k_i}$ with $k_i\in \poi$ be a minimal subpath of $L_i$ such that every $s\in S_{\mf c}$ has a neighbor in $L_i'$ and let $s(i)$ be a vertex of $S_{\mf c}$ such that its only neighbor in $L_i'$ is $v_i^1$.
	By the choice of $l$, there exists a vertex $s_1$ and a set $I\subseteq \poi_l$ of size $c!+1$ such that $s_1=s(i)$ for every $i\in I$.
	
	For every $i\in I$ and $s\in S_{\mf c}$, let $j=j(i,s)\in \poi $ be minimum such that $s$ is adjacent to $v_i^j$, i.e., $j(i,s)$ is the index of the first neighbor of $s$ in $L_i'$.
	Because $\mf c$ is ample and $|I| = c!+1$, there are distinct $q,w\in I$ and an ordering $s_1,\dots,s_{c+1}$ of the vertices of $S_{\mf c}$ such that $j(q,s_a)<j(q,s_b)$ and $j(w,s_a)<j(w,s_b)$ for $a<b$, i.e., the first neighbors of $s_1,\dots,s_{c+1}$ appear in the same order in both $L_q'$ and $L_w'$.
	Finally, observe that the cycles
	\[
		C_i = s_1v_q^1
		\cup \bigl(v_q^1 \dd L_q' \dd v_q^{j(q,s_i)}\bigr)
		\cup v_q^{j(q,s_i)}s_i
		\cup s_iv_w^{j(w,s_i)}
		\cup \bigl(v_w^{j(w,s_i)} \dd L_w' \dd v_w^1\bigr)
		\cup v_w^1s_1,
	\]
	for $i=2,\dots,c+1$ form $c$ induced cycles of increasing lengths.
	Thus, $\cl(G)\ge c$ as required.
	
\end{proof}

\subsection{Handling domes}

\starDefinition

\domeDefinition

In this subsection, we prove Lemma \ref{lem:alignment_has_cycles}, which states that if a graph $G$ contains a large enough dome, then it also contains induced cycles of many distinct lengths.
As a result, the proof of Theorem \ref{thm:main_result} will be reduced to finding large domes in graphs.
We start by introducing a lemma, which follows immediately from Lemma 4.4 in \cite{tw17}.

\begin{lemma}[Chudnovsky, Hajebi, Spirkl \cite{tw17}]
\label{lem:get_an_alignment}

	For all $d,s,t\in \poi$, there is a constant $f_{\blackref{lem:get_an_alignment}}=f_{\blackref{lem:get_an_alignment}}(d,s,t)$ such that every graph $G$ with an $(f_{\blackref{lem:get_an_alignment}}, d)$-dome contains
	\begin{enumerate}[\rm (a)]
		\item an aligned $(s,1)$-dome; or
		\item an induced $K_{t,t}$-minor.
	\end{enumerate}

\end{lemma}

We prove:

\begin{lemma}
\label{lem:alignment_has_cycles}

	For all $c,d,t \in \poi$, there exists a constant $f_{\blackref{lem:alignment_has_cycles}}=f_{\blackref{lem:alignment_has_cycles}}(c,d,t)$ such that if a graph $G$ contains an $(f_{\blackref{lem:alignment_has_cycles}},d)$-dome and does not contain an induced $K_{t,t}$-minor, then $\cl(G) \ge c$.

\end{lemma}

\begin{proof}

	Define $s = c^2+1$ and $f_{\blackref{lem:alignment_has_cycles}} = f_{\blackref{lem:get_an_alignment}}(d,s,t)$.
	By Lemma \ref{lem:get_an_alignment}, $G$ contains an aligned $(s,1)$-dome $(S,P)$.
	Let $x$ be the root of $S$ and let $L=\{l_1,\dots, l_s\}$ be its leaf set such that $P$ traverses $N(l_1)\cap P,\dots,N(l_s)\cap P$ in this order.
	For $i\in \poi_s$, let $S_i$ be the $x\dd l_i$ path in $S$.
	
	By Theorem \ref{thm:erdos_szekeres}, there is a subsequence of $c+1$ paths $S_{i_1},\dots,S_{i_{c+1}}$ preserving the ordering along $P$ of either nondecreasing or nonincreasing lengths.
	We may assume that the lengths are nondecreasing by reindexing if needed.
	For $j=2,\dots,c+1$, let $P_{i_j}$ be the subpath of $P$ such that $l_{i_1}\dd P_{i_j} \dd l_{i_j}$ is a path.
	Note that the paths $P_{i_2},\dots,P_{i_{c+1}}$ have strictly increasing lengths.
	Therefore, $x\dd S_{i_1} \dd l_{i_1} \dd P_{i_j}\dd l_{i_j}\dd S_{i_j} \dd x$ for $j=2,\dots,c+1$ are induced cycles in $G$ of strictly increasing lengths.
	It follows that $\cl(G)\ge c$ as claimed.
	
\end{proof}

\section{Reducing to a banana}
\label{sec:reducing_to_a_banana}

In this section, we derive Theorem \ref{thm:main_result} using lemmas from Section \ref{sec:obstructions}, Theorem \ref{thm:banana_has_cycles} stated shortly, and two results from \cite{tw7} and \cite{tw17}.
We will prove Theorem \ref{thm:banana_has_cycles} in Sections \ref{sec:analyzing_bananas}, \ref{sec:banana_to_kite} and \ref{sec:general_step}.

We start with a couple of definitions.
Let $G$ be a graph.
We say that $G$ is $t$-\textit{tidy} if $G$ is $K_{t+1}$-free and has no induced $K_{t,t}$-minor.

\bananaDefinition

\begin{restatable}{theorem}{bananaHasCycles}
\label{thm:banana_has_cycles}

	For all $l,s,t\in \poi$, there are constants $f_{\blackref{thm:banana_has_cycles}}=f_{\blackref{thm:banana_has_cycles}}(l,s,t)$ and $g_{\blackref{thm:banana_has_cycles}}=g_{\blackref{thm:banana_has_cycles}}(t)$ such that for every $t$-tidy graph $G$ that contains an $l$-dense $f_{\blackref{thm:banana_has_cycles}}$-banana $B_{x,z}$, there is an $(s,g_{\blackref{thm:banana_has_cycles}})$-dome with root $x$ in $G$.
	
\end{restatable}

Let $k, m \in \poi$ and let $G$ be a graph.
A $(k, m)$-\textit{block $B$ in $G$} is a pair $(X, \mca{P})$ where $X \subseteq V(G)$ and $\mca{P}$ is a set of paths with ends in $X$ such that 
\begin{itemize}
	\item $|X| = k$;
	\item for all distinct $x,y\in X$, there is a set $\mca{P}_{\{x, y\}} \subseteq \mca P$ of $m$ pairwise internally disjoint $x\dd y$ paths;
	\item $\mca P = \cup_{x\neq y\in X} \mca{P}_{\{x, y\}}$.
\end{itemize}
We say that $B$ is \textit{strong} if for all distinct pairs $\{x, y\},\left\{x^{\prime}, y^{\prime}\right\}$ of $X$, we have $V\left(\mca{P}_{\{x, y\}}\right) \cap V\left(\mca{P}_{\left\{x^{\prime}, y^{\prime}\right\}}\right)=\{x, y\} \cap\left\{x^{\prime}, y^{\prime}\right\}$; that is, every path $P \in \mca{P}_{\{x, y\}}$ is disjoint from every path $P^{\prime} \in \mca{P}_{\left\{x^{\prime}, y^{\prime}\right\}}$, except $P$ and $P^{\prime}$ may share an end.
We call $B$ \textit{anticomplete} if for every pair of distinct vertices $x,y\in X$, the paths in $\mca{P}_{\{x, y\}}^*$ are pairwise anticomplete.
We will need the following two results about blocks.
Theorem \ref{thm:clean_graphs_have_strong_blocks} is immediate from Theorem 6.5 in \cite{tw7} and Lemma \ref{lem:anti_block_has_obstructions} is immediate from Lemma 3.6 in \cite{tw17}.

\begin{theorem}[Abrishami, Alecu, Chudnovsky, Hajebi, Spirkl \cite{tw7}]
\label{thm:clean_graphs_have_strong_blocks}

	For all $k, l, t \in \poi$, there is a constant $f_{\blackref{thm:clean_graphs_have_strong_blocks}} = f_{\blackref{thm:clean_graphs_have_strong_blocks}}(k,l,t) \in \poi$ with the following property.
	Let $G$ be a graph that does not contain $K_{t+1},K_{t,t}$, a subdivision of $W_{t\times t}$, or the line graph of a subdivision of $W_{t\times t}$.
	If $\tw(G) \ge f_{\blackref{thm:clean_graphs_have_strong_blocks}}$, then $G$ contains a strong $(k,l)$-block.

\end{theorem}

\begin{lemma}[Chudnovsky, Hajebi, Spirkl \cite{tw17}]
\label{lem:anti_block_has_obstructions}

	For all $s, t \in \poi$, there are constants $f_{\blackref{lem:anti_block_has_obstructions}}=f_{\blackref{lem:anti_block_has_obstructions}}(s, t) \in \poi$ and $g_{\blackref{lem:anti_block_has_obstructions}}=g_{\blackref{lem:anti_block_has_obstructions}}(s) \in \poi$ with the following property.
	Let $G$ be a $K_{t+1}$-free graph and let $B$ be a strong anticomplete $(f_{\blackref{lem:anti_block_has_obstructions}}, g_{\blackref{lem:anti_block_has_obstructions}})$-block in $G$.
	Then $G$ contains
	\begin{enumerate}[\rm (a)]
		\item a proper subdivision of $K_s$, or
		\item an induced $K_{s,s}$-minor.
	\end{enumerate}

\end{lemma}

We are now ready to prove Theorem \ref{thm:main_result}, which we restate.

\mainResult*

\allowdisplaybreaks
\begin{proof}

	Define
	\begin{align*}
		t' &= \max \{t, f_{\blackref{lem:wall_has_cycles}}(c), f_{\blackref{lem:Ktt_minor_has_cycles}}(c,t)\}\\
		k &= f_{\blackref{lem:anti_block_has_obstructions}}(t', t')\\
		l &= g_{\blackref{lem:anti_block_has_obstructions}}(t')\\
		d & = g_{\blackref{thm:banana_has_cycles}}(t')\\
		s & = f_{\blackref{lem:alignment_has_cycles}}(c,d,t')\\
		q &= f_{\blackref{thm:banana_has_cycles}}(l,s,t')\\
		f_{\blackref{thm:main_result}} &= f_{\blackref{thm:clean_graphs_have_strong_blocks}}(k,q,t')
	\end{align*}
	Let $G$ be a graph with $\tw(G) \ge f_{\blackref{thm:main_result}}$.
	Suppose that $G$ is $\{K_{t+1},K_{t,t}\}$-free and that $\cl(G) < c$.
	By the choice of $t'$, Lemmas \ref{lem:wall_has_cycles} and \ref{lem:Ktt_minor_has_cycles} imply that $G$ does not contain 
	\begin{itemize}
		\item $K_{t'+1}$,
		\item $K_{t',t'}$,
		\item a subdivision of $W_{t'\times t'}$,
		\item the line graph of a subdivision of $W_{t'\times t'}$,
		\item a proper subdivision of $K_{t'}$, nor
		\item an induced $K_{t',t'}$-minor.
	\end{itemize}
	Then, by Theorem \ref{thm:clean_graphs_have_strong_blocks}, $G$ has a strong $(k,q)$-block $B=(X,\mca P)$.
	By the observations above and Lemma \ref{lem:anti_block_has_obstructions}, $G$ does not contain a strong anticomplete $(k,l)$-block.
	Thus, $G$ contains an $l$-dense $q$-banana $B'$.
	By Theorem \ref{thm:banana_has_cycles}, $G$ contains an $(s,d)$-dome $D$.
	Finally, by Lemma \ref{lem:alignment_has_cycles}, it follows that $\cl(G)\ge c$, which is a contradiction.
	This completes the proof of the theorem.

\end{proof}

\section{Analyzing bananas}
\label{sec:analyzing_bananas}

In this section, we analyze the structure of rigid bananas and prove Lemmas \ref{lem:banana_to_theta} and \ref{lem:degree_two} that will be useful for us in Section \ref{sec:banana_to_kite}.
We start with a few definitions and Theorem \ref{thm:connectifier} from \cite{tw10}.

A vertex $v$ in a graph $G$ is said to be a \textit{branch vertex} if $v$ has degree more than two.
By a \textit{comb} we mean a tree $C$ with maximum degree three such that there exists a path $P$ of $C$ where all branch vertices of $C$ belong to $P$.
We call a maximal such path $P$ a \textit{spine} of $C$.
Note that in particular, the ends of a spine of $C$ have degree one in $C$.
(The definition of a comb is the definition of a {\em caterpillar} from \cite{tw10}.)

For a graph $H$, a vertex $v$ of $H$ is said to be \textit{simplicial} if $N_H(v)$ is a clique.
We denote by $\mca{Z}(H)$ the set of all simplicial vertices of $H$.
Note that for every tree $T$ with at least two vertices, $\mca{Z}(T)$ is the set of all leaves of $T$.
An edge $e$ of a tree $T$ is said to be a \textit{leaf-edge} of $T$ if $e$ is incident with a leaf of $T$.
It follows that if $H$ is the line graph of a tree $T$, then $\mca{Z}(H)$ is the set of all vertices in $H$ corresponding to the leaf-edges of $T$.
We will need the following theorem, which is immediate from Theorem 5.4 in \cite{tw10}.

\begin{theorem}[Abrishami, Alecu, Chudnovsky, Hajebi, Spirkl \cite{tw10}]
\label{thm:connectifier}

	For every integer $h \geq 1$, there exists an integer $f_{\blackref{thm:connectifier}}=f_{\blackref{thm:connectifier}}\left(h\right) \geq 1$ with the following property.
	Let $G$ be a connected graph with no clique of cardinality $h$.
	Let $S \subseteq V(G)$ such that $\left|S\right| \geq f_{\blackref{thm:connectifier}}, G \setminus S$ is connected, and every vertex of $S$ has a neighbor in $G \setminus S$.
	Then there is a set $S' \subseteq S$ with $|S'|=h$ and an induced subgraph $H$ of $G \setminus S$ for which one of the following holds.
	\begin{enumerate}[\rm (i)]
		\item \label{connectifier_a} $H$ is a path and every vertex of $S'$ has a neighbor in $H$; or
		\item \label{connectifier_b} $H$ is a comb, or the line graph of a comb, or a subdivided star.
	Moreover, every vertex of $S'$ has a unique neighbor in $H$ and $H \cap N(S')=\mca{Z}\left(H\right)$.
	\end{enumerate}

\end{theorem}

For convenience, let us restate the definition of thetas.
We define a \textit{$q$-theta in $G$ with ends $x,y$}, denoted $\Theta_{x,y}$, to be a $q$-banana $(\{x,y\},\{P_1,\dots,P_q\})$ in $G$ such that for all distinct $i,j \in \poi_q$, the paths $P_i\setminus \{x,y\}$ and $P_j\setminus \{x,y\}$ are anticomplete to each other (see Figure \ref{fig:banana}).

We will now apply Theorem \ref{thm:connectifier} to show how to carve a theta from a rigid banana.
First, we do it with the additional assumption that all the neighbors of one of the ends of the banana have degree two.

\begin{lemma}
\label{lem:banana_to_theta}

	Let $q,s,t \in \poi$ and let $G$ be a $t$-tidy graph.
	There exists a constant $f_{\blackref{lem:banana_to_theta}} = f_{\blackref{lem:banana_to_theta}}(q,s,t)$ such that if $G$ contains a rigid $f_{\blackref{lem:banana_to_theta}}$-banana $B_{x,z}=(\{x,z\},\mca P)$ such that every vertex in $N(x)\cap V(B_{x,z})$ has degree two in $V(B_{x,z})$, then one of the following holds.
	\begin{enumerate}[\rm (a)]
		\item $G\setminus \{z\}$ contains a $q$-theta $(\{x,y\},\mca Q)$.
		\item $G$ contains an $(s,2)$-dome with root $x$.
	\end{enumerate}

\end{lemma}

\begin{proof}

	Define $h = \max\{q,s,t+1\}$ and $f_{\blackref{lem:banana_to_theta}} = f _{\blackref{thm:connectifier}}(h)$.

	Recall that $G$ is $t$-tidy and observe that $S=N(x)\cap V(B_{x,z})$ is a stable set in $B_{x,z}$.
	Also, observe that $G'= G[V(B_{x,z}^*) \setminus S]$ is connected (since $B_{x,z}$ is rigid and every vertex in $S$ has degree 1 in $G[V(B_{x,z}^*)]$) and that every vertex in $S$ has a neighbor in $G'$.
	Applying Theorem \ref{thm:connectifier}, we obtain a set $S'\subseteq S$ of size $h$ and an induced subgraph $H$ of $G'$ for which one of the following holds.
	\begin{enumerate}[\rm (i)]
		\item $H$ is a path and every vertex of $S'$ has a neighbor in $H$; or
		\item $H$ is a comb, or the line graph of a comb, or a subdivided star.
			Moreover, every vertex of $S'$ has a unique neighbor in $H$ and $H \cap N(S')=\mca{Z}\left(H\right)$.
	\end{enumerate}
	Define $\tilde G = G[\{x\} \cup S' \cup V(H)]$.
	Suppose that case \blackref{connectifier_a} holds.
	Recall that every vertex of $S$ has degree two in $B_{x,z}$.
	Since each $s\in S$ lies on a distinct $x\dd z$ path in $\mca P$, it follows that no two vertices of $S$ share a neighbor in $H$.
	Hence, every vertex of $H$ has at most one neighbor in $S'$.
	Therefore, $\tilde G$ contains an $(s,1)$-dome with root $x$ (which by extension is also an $(s,2)$-dome) as claimed.

	Thus, we may assume that case \blackref{connectifier_b} holds.
	Note that $h\ge q$.
	Suppose that $H$ is a subdivided star with root $r$.
	Then $\tilde G$ contains a $q$-theta $\Theta_{x,r}$ as claimed.
	Since $V(H) \subseteq V(B_{x,z}^*)$, it is clear that $\Theta_{x,r}$ does not contain $z$.
	
	Next, suppose that $H$ is a comb and let $P$ be a spine of $H$.
	Let $L=N(P)\cap (S'\cup V(H))$.
    Observe that every vertex of $P$ has at most one neighbor in $L$.
	Since $H \cap N(S')=\mca{Z}\left(H\right)$, we have that $G[\{x\} \cup S' \cup (V(H)\setminus V(P))]$ is a subdivided star in $G$ with leaf set $L$ where $|L|\ge |S'|\ge s$.
	It follows that every vertex of $L$ has a unique neighbor in $P$, and $x$ is anticomplete to $P$.
	Thus, $G$ contains an $(s,1)$-dome with root $x$ as claimed.

	Therefore, suppose that $H$ is the line graph of a comb $C$.
	Let $P$ be a spine of $C$ and let $P'$ be the path in $H$ corresponding to $P$.
	Let $L=N(P')\cap (S'\cup V(H))$.
	Observe that every vertex of $P'$ has at most two neighbors in $L$.
	Since $H \cap N(S')=\mca{Z}\left(H\right)$, we have that $G[\{x\} \cup S' \cup (V(H)\setminus V(P'))]$ is a subdivided star in $G$ with leaf set $L$ where $|L|\ge |S'|\ge s$.
	It follows that every vertex of $L$ has a neighbor in $P'$, and $x$ is anticomplete to $P'$.
	Thus, $G$ contains an $(s,2)$-dome with root $x$ as claimed.
	This completes the proof.
	
\end{proof}

Next, we justify the additional assumption of Lemma \ref{lem:banana_to_theta}, but first we introduce some definitions that we will need.

We define a \textit{directed graph} to be a pair $(V,D)$ where $D\subseteq V\times V$ is the set of directed edges.
As before, we do not allow loops or parallel edges.
A \textit{tournament} is a directed graph $G=(V,D)$ such that for all distinct $u,v\in V$ exactly one of $(u,v),(v,u)$ is in $D$.
A tournament is called \textit{transitive} if $(u,v),(v,w)\in D$ implies that $(u,w)\in D$.
The \textit{underlying graph of} a directed graph $G$ is the graph $G^-$ with the same vertex set and an edge $uv$ if and only if $(u,v)\in D$ or $(v,u)\in D$.
A well-known fact about tournaments is the following theorem.

\begin{theorem} [Stearns; see Erd\H{o}s, Moser \cite{tournament}]
\label{thm:tournament}

	Every tournament on $2^{n-1}$ vertices contains a transitive tournament on $n$ vertices.

\end{theorem}

We will also use the following classical result.

\begin{theorem} [Ramsey \cite{Ramsey}]
\label{thm:ramsey}

	For all $s,t \in \poi$, every graph on at least $s^t$ vertices has either a stable set of cardinality $s$ or a clique of cardinality $t+ 1$.

\end{theorem}

We are now ready to prove Lemma \ref{lem:degree_two}.

\begin{lemma}
\label{lem:degree_two}

	Let $q,s,t \in \poi$ and let $G$ be a $t$-tidy graph.
	There exists a constant $f_{\blackref{lem:degree_two}} = f_{\blackref{lem:degree_two}}(q,s,t)$ such that if $G$ contains a rigid $f_{\blackref{lem:degree_two}}$-banana $B_{x,z}=(\{x,z\},\mca P)$, then one of the following holds.
	\begin{enumerate}[\rm (a)]
		\item $G$ contains a rigid $q$-banana $B'_{x,z}=(\{x,z\},\mca P')$ such that every vertex in $N(x)\cap V(B'_{x,z})$ has degree two in $B'_{x,z}$.
		\item $G$ contains an $(s,t)$-dome with root $x$.
	\end{enumerate}

\end{lemma}

\begin{proof}
	
	Define
	\allowdisplaybreaks
	\begin{align*}
		r &= {\max\{s,t\}\choose t}t^t\\
		w &= 2^{s+r-1}\\
		f &= q^{w-1}\\
		f_{\blackref{lem:degree_two}} &= f^t
	\end{align*}
	
	Suppose that $G$ does not contain an $(s,t)$-dome with root $x$.
	Let $B_{x,z}=(\{x,z\},\mca P)$ be a rigid $f_{\blackref{lem:degree_two}}$-banana in $G$ and let $X=N(x)\cap V(B_{x,z})$.
	Since $G$ is $t$-tidy, Theorem \ref{thm:ramsey} implies that there exists a stable set $X'=\{x_1,\dots,x_f\}\subseteq X$ of size $f$.
	Let $\mca P' =\{P_1,\dots,P_f\}\subseteq \mca P$ be such that $P_i\setminus \{x\}$ is an $x_i\dd z$ path for each $i\in \poi_f$.
	Define a directed graph $H$ with vertex set $\poi_f$ where $(i,j)\in E(H)$ for $i\neq j$ if and only if $x_i$ has a neighbor in $P_j^*\setminus \{x_j\}$.
	First, we claim that
		
	\sta{\label{sta:degree_two_clique} $H^-$ does not contain a clique of size $w$.}
	
	Suppose that $C$ is a clique in $H^-$ of size $w$.
	Then $H[C]$ contains a subgraph $H'$ isomorphic to a tournament on $w$ vertices.
	By Theorem \ref{thm:tournament}, $H'$ contains a transitive tournament on $s+r$ vertices, say $\{a_1,\dots,a_s,b_1,\dots,b_r\}$ where $(a_i,b_j)$ is an edge for all $i\in \poi_s,j\in \poi_r$.
	Recall that this means that $x_{a_i}$ has a neighbor in $P_{b_j}^*\setminus \{x_{b_j}\}$.
	Suppose that there exists a $j \in \poi_r$ such that every vertex of $P_{b_j}^*\setminus \{x_{b_j}\}$ has at most $t$ neighbors in $X''=\{x_{a_1},\dots,x_{a_s}\}$.
	Note that $G[\{x\}\cup X'']$ is a star in $G$ with root $x$ and that every vertex in $X''$ has a neighbor in $P_{b_j}^*\setminus \{x_{b_j}\}$.
	Since $x$ is anticomplete to $P_{b_j}^*\setminus \{x_{b_j}\}$, it follows that $G$ contains an $(s,t)$-dome with root $x$, which is a contradiction.
	
	Therefore, for every $j \in \poi_r$, $P_{b_j}^*\setminus \{x_{b_j}\}$ contains a vertex $v_j$ that has at least $t$ neighbors in $X''$.
	Since $r = {s\choose t}t^t$, there exist sets $A \subseteq X''$ and $B\subseteq \{v_1,\dots,v_r\}$ such that $|A|=t, |B|=t^t$, and $ab$ is an edge for all $a\in A,b\in B$.
	Finally, recall that $G$ is $t$-tidy.
	By Theorem \ref{thm:ramsey}, $B$ contains a stable set $B'$ of size $t$.
	However, $G[A\cup B']$ is isomorphic to $K_{t,t}$, which is a contradiction.
	This completes the proof of \eqref{sta:degree_two_clique}.
	
	\
	
	Theorem \ref{thm:ramsey} and (\ref{sta:degree_two_clique}) imply that $H^-$ contains a stable set $S$ of size $q$, say $S=\{s_1,\dots,s_q\}$.
	Let $\mca P''=\{P_{s_1},\dots,P_{s_q}\}$ and $B'_{x,z}=(\{x,z\},\mca P'')$.
	Note that for all distinct $i,j\in \poi_q$, $x_{s_i}$ does not have a neighbor in $P_{s_j}^*$.
	Thus, $x_{s_i}$ has degree two in $B'_{x,z}$ as required.
	This completes the proof.
	
\end{proof}

\section{From a banana to a kite}
\label{sec:banana_to_kite}

In this section, we prove Lemma \ref{lem:banana_to_kite} that shows how to obtain a kite from a rigid dense banana and then show how to make this kite clean in Theorem \ref{thm:banana_to_clean_kite}.
For convenience, let us restate the definitions of tailed thetas and kites.

\tailedThetaDefinition

\kiteDefinition

\begin{lemma}
\label{lem:banana_to_kite}

	For all $k,l,q,s,t \in \poi$, there are constants $f_{\blackref{lem:banana_to_kite}} = f_{\blackref{lem:banana_to_kite}}(k,l,q,s,t)$ and $g_{\blackref{lem:banana_to_kite}} = g_{\blackref{lem:banana_to_kite}}(t)$ with the following property.
	Let $G$ be a $t$-tidy graph that contains a rigid $l$-dense $f_{\blackref{lem:banana_to_kite}}$-banana $B_{x,z}$.
	Then $G$ contains an $(x,z)$-tailed $(k,q)$-kite or an $(s,g_{\blackref{lem:banana_to_kite}})$-dome with root $x$.

\end{lemma}

\begin{proof}

	Define 
	\begin{align*}
		p &= f_{\blackref{lem:banana_to_theta}}(2q+l,s,t)\\
		f_{\blackref{lem:banana_to_kite}} &= f_{\blackref{lem:degree_two}}(pk,s,t)\\
		g_{\blackref{lem:banana_to_kite}} &= \max \{2, t\}
	\end{align*}

	Suppose that $G$ does not contain an $(s,g_{\blackref{lem:banana_to_kite}})$-dome with root $x$.
	Lemma \ref{lem:degree_two} implies that $G$ contains a rigid $pk$-banana $B_{x,z}=(\{x,z\},\mca P)$ such that every vertex in $N(x)\cap V(B_{x,z})$ has degree two in $B_{x,z}$.
	Partition $\mca P$ into $k$ sets $\mca P_1,\dots,\mca P_k$ of size $p$ each.
	For $i\in \poi_k$, let $B_i=(\{x,z\},\mca P_i)$ be the resulting banana.
	Apply Lemma \ref{lem:banana_to_theta} to every $B_i$ to obtain a $(2q+l)$-theta $\Theta_{x,y_i} =(\{x,y_i\}, \mca Q_i)$.
	Suppose that $z$ has neighbors in at least $l$ paths of $(\mca Q_i)^*$.
	Then, $B_{x,z}$ contains $l$ pairwise internally vertex-disjoint $x\dd z$ paths whose interiors are pairwise anticomplete, or in other words an $l$-theta with ends $x,z$, contradicting $B_{x,z}$ being $l$-dense.
	Thus, we can remove from $(\mca Q_i)^*$ the paths that are not anticomplete to $z$ to obtain a $2q$-theta $\Theta_{x,y_i}' =(\{x,y_i\}, {\mca Q}_i')$ such that $z$ is anticomplete to $V(({\mca Q}_i')^*)$.
	We claim that

	\sta{\label{sta:can_find_a_tail} For $i\in \poi_k$, there is an $(x,z)$-tailed $q$-theta $\Theta_i$ in $B_i$.}
	
	Since $V(\Theta_{x,y_i}')\subseteq V(\mca P_i)$, there exists a path $P_i$ in $\mca P_i$ that contains $y_i$.
	Let $v_i$ be the first vertex of $P_i$ with a neighbor in $V((\Theta_{x,y_i}')^*)$, traversing $P_i$ from $z$ to $x$.
	Such a vertex $v_i$ exists because $y_i$ is one possible candidate.
	Note that $v_i\neq z$ since $z$ is anticomplete to $V((\mca Q_i')^*)$.
	Suppose that $v_i$ has a neighbor in at least $q$ paths of $({\mca Q}_i')^*$, say $Q_{i,1},\dots,Q_{i,q}$.
	For $j\in \poi_q$, let $Q_{i,j}'$ be a path from $x$ to $v_i$ with interior in $Q_{i,j}$.
	Then $\Theta_i=((\{x,v_i\},\{Q_{i,j}' \mid j\in \poi_q\}), v_i \dd P_i \dd z)$ is an $(x,z)$-tailed $q$-theta in $B_i$.

	Thus, we may assume that there is a set of $q$ paths $\mca Q_i'' \subseteq \mca Q_i'$ such that $v_i$ is anticomplete to $V((\mca Q_i'')^*)$.
	It follows that $v_i\neq y_i$.
	Let $Q_i$ be some path in $\mca Q_i'$ such that $v_i$ has a neighbor in $V(Q_i^*)$.
	Such a path exists by the definition of $v_i$.
	Let $Q_i'=y_i\dd Q_i \dd v_i$ be a path from $y_i$ to $v_i$ with interior in $Q_i$.
	Since every vertex in $N(x)\cap V(B_i)$ has degree two in $B_i$, it follows that $(Q_i')^*$ is anticomplete to $x$.
	We also have that $v_i$ and $y_i$ are anticomplete to $x$ since they both have two neighbors in $B_i$ distinct from $x$.
	Thus, $Q_i'$ is anticomplete to $x$.
	Also, note that $Q_i'\setminus \{y_i\}$ is anticomplete to $V((\mca Q_i'')^*)$ (since paths in $(\mca Q_i')^*$ are anticomplete as they are part of a theta).
	Thus, $\mca Q_i''$ together with the tail $y_i \dd Q_i' \dd v_i \dd P_i \dd z$ form an $(x,z)$-tailed $q$-theta in $B_i$ as required.
	This completes the proof of \eqref{sta:can_find_a_tail}.

	\

	Since $\mca P_1,\dots,\mca P_k$ form a partition of $\mca P$, we have that for all distinct $i,j \in \poi_k$, $V(\Theta_i) \cap V(\Theta_j) = \{x,z\}$.
	Thus, $\allowbreak \Theta_1,\dots, \Theta_k$ form an $(x,z)$-tailed $(k,q)$-kite $\mca K$ in $G$ as required.
	This completes the proof.

\end{proof}

Let $G$ be a complete graph.
A {\em $k$-coloring} of $G$ is a function $c:E(G)\to S$ where $S$ is some set of size $k$.
A {\em monochromatic} clique in $G$ is a clique $C$ in $G$ such that $c(e)=c(e')$ for all $e,e'\in E(C)$.
We will need the following multi-color analogue of Theorem \ref{thm:ramsey}.

\begin{theorem}[Ramsey \cite{Ramsey}]
\label{thm:multicolor-ramsey}
	
	Let $S=\{s_1,\dots,s_k\}$ be a set and $n_1,\dots,n_k\in \poi$ be integers.
	There exists an integer $R_{\blackref{thm:multicolor-ramsey}}=R_{\blackref{thm:multicolor-ramsey}}(n_1,\dots,n_k)$ such that for every $k$-coloring $c: E(K_{R_{\blackref{thm:multicolor-ramsey}}})\to S$ of $K_{R_{\blackref{thm:multicolor-ramsey}}}$ there exists $i\in \poi_k$ and a monochromatic clique in color $s_i$ of size $n_i$.
	
\end{theorem}

We are now ready to prove the main result of this section.

\begin{theorem}
\label{thm:banana_to_clean_kite}

	For all $k,l,q,s,t \in \poi$, there are constants $f_{\blackref{thm:banana_to_clean_kite}} = f_{\blackref{thm:banana_to_clean_kite}}(k,l,q,s,t)$ and $g_{\blackref{thm:banana_to_clean_kite}} = g_{\blackref{thm:banana_to_clean_kite}}(t)$ with the following property.
	Let $G$ be a $t$-tidy graph that contains a rigid $l$-dense $f_{\blackref{thm:banana_to_clean_kite}}$-banana $B_{x,z}$.
	Then $G$ contains a clean $(x, z)$-tailed $(k,q)$-kite or an $(s,g_{\blackref{thm:banana_to_clean_kite}})$-dome with root $x$.
	
\end{theorem}

\begin{proof}

	Define
	\allowdisplaybreaks
	\begin{align*}
		q' &= q+ks \\
		\theta &= {\max\{t,q'\} \choose t}t\\
		\gamma &= \theta^t\\
		\psi &= 2^\gamma \\
		k' &= k^{\psi-1} \\
		k'' &= (k')^{t} \\
		q'' &= R_{\blackref{thm:multicolor-ramsey}}(k''q',\underbrace{2t,\ldots,2t}_{(k'')^2-k''\text{ times}})\\
		f_{\blackref{thm:banana_to_clean_kite}} &= f_{\blackref{lem:banana_to_kite}}(k'',l, q'', s,t)\\
		g_{\blackref{thm:banana_to_clean_kite}} &= \max \{2,t\}
	\end{align*}

	Suppose that $G$ does not contain an $(s,g_{\blackref{thm:banana_to_clean_kite}})$-dome with root $x$.
	Lemma \ref{lem:banana_to_kite} implies that $G$ contains an $(x,z)$-tailed $(k'', q'')$-kite $\mca K_0 = \{(\Theta_{x,y_i}^0,P_i) \mid i \in \poi_{k''}\}$.
	We will clean up the kite $\mca K_0$ in three steps.
	
	\begin{enumerate}

		\item First, we will obtain an $(x,z)$-tailed $(k'', q')$-kite $\mca K_1 = \{(\Theta_{x,y_i}^1,P_i) \mid i \in \poi_{k''}\}$ such that $V((\Theta_{x,y_i}^1)^*)$ and $V((\Theta_{x,y_j}^1)^*)$ are pairwise anticomplete for $i\neq j$.
		\item Then, we will obtain an $(x,z)$-tailed $(k', q')$-kite $\mca K_2 = \{(\Theta_{x,y_i}^2,P_i) \mid i \in \poi_{k'}\}$ such that additionally, the set $\{y_1,\dots,y_{k'}\}$ is stable.
		\item Lastly, we will obtain an $(x,z)$-tailed $(k, q)$-kite $\mca K_3 = \{(\Theta_{x,y_i}^3,P_i) \mid i \in \poi_{k}\}$ such that we also have that $V((\Theta_{x,y_i}^3)^*)$ and $P_j$ are anticomplete for $i\neq j$.

	\end{enumerate}
	
	We will then have that for all $i\in\poi_k$ every vertex in $V((\Theta_{x,y_i}^3)^*)$ has degree two in $\mca K_3$ and that the set $\{y_1,\dots,y_k\}$ is stable, which implies that $\mca K_3$ is clean as desired.
	
	\
	
	{\em Step 1.}

	\
	
	For every $i\in \poi_{k''}$, let $\Theta_{x,y_i}^0=\{Q_1^i,\dots,Q_{q''}^i\}$.
	Let $H$ be a complete graph with vertex set $\poi_{q''}$.
	Color the edges of $H$ the following way.
	For all $x,y\in \poi_{q''}$ with $x<y$ and all distinct $i,j\in \poi_{k''}$, if $xy\in E(H)$ has not yet been colored and $(Q_x^i)^*$ is not anticomplete to $(Q_y^j)^*$, then color $xy$ with color $(i,j)$.
	Assign every uncolored edge of $H$ the color $0$.
	
	Suppose that $H$ contains a monochromatic clique $C$ of color $(i,j)$ for $i\neq j\in \poi_{k''}$ with $|C|= 2t$.
	After possible relabeling, we may assume that $C=\poi_{2t}$.
	Then we have that $(Q_1^i)^*,\dots,(Q_t^i)^*$ and $(Q_{t+1}^j)^*,\dots,(Q_{2t}^j)^*$ form an induced $K_{t,t}$-minor, contradicting $G$ being $t$-tidy.
	By Theorem \ref{thm:multicolor-ramsey}, it follows that $H$ contains a monochromatic clique $C$ of color $0$ with $|C|= k''q'$.
	Let $C_1\cup\dots \cup C_{k''}$ be a partition of $C$ such that $|C_i|=q'$ for every $i\in \poi_{k''}$.
	For every $i\in \poi_{k''}$, let $\Theta_{x,y_i}^1=\{Q_x^i\mid x\in C_i\}$.
	It follows that $V((\Theta_{x,y_i}^1)^*)$ and $V((\Theta_{x,y_j}^1)^*)$ are pairwise anticomplete for $i\neq j$ as required.
	This completes Step 1.
	
	\
	
	{\em Step 2.}
	
	\
	
	Let $\mca K_1 = \{(\Theta_{x,y_i}^1,P_i) \mid i \in \poi_{k''}\}$.
	Since $G$ is $t$-tidy, Theorem \ref{thm:ramsey} implies that there is a set $I\subseteq \poi_{k''}$ of size $k'$ such that $\{y_i\mid i\in I\}$ is a stable set.
	By renumbering if necessary, we may assume that $I = \poi_{k'}$.
	For every $i\in\poi_{k'}$, let $\Theta_{x,y_i}^2=\Theta_{x,y_i}^1$, and let $\mca K_2 = \{(\Theta_{x,y_i}^2,P_i) \mid i \in \poi_{k'}\}$.
	This completes Step 2.

	\
	
	{\em Step 3.}
	
	\
	
	Consider a directed graph $D$ with vertex set $\poi_{k'}$ and an edge from $i$ to $j$ for $i\neq j$ if and only if $P_i$ has neighbors in at least $s$ paths of $(\Theta^2_{x,y_j})^*$.
	We claim:

	\sta{\label{sta:cleanup_complete} $D^-$ does not contain a clique of size $\psi$.}

	Suppose that $C$ is a clique of size $\psi$ in $D^-$.
	Then $D[C]$ contains a subgraph $H'$ isomorphic to a tournament of size $\psi$.
	By Theorem \ref{thm:tournament}, $H'$ contains a transitive tournament $T$ of size $\gamma+1$.
	Therefore, there exists a set of $\gamma$ paths $\mca P=\{P_{i_1},\dots,P_{i_{\gamma}}\}$ and a theta $\Theta^2_{x,y_j}$ such that every path in $\mca P$ has a neighbor in at least $s$ paths of $(\Theta^2_{x,y_j})^*$.

	Suppose that there exists a path $P_i\in \mca P$ such that every vertex $v$ of $P_i$ has neighbors in at most $t$ paths of $(\Theta^2_{x,y_j})^*$.
	Then there is an $(s,g_{\blackref{thm:banana_to_clean_kite}})$-dome in $G[\{x\} \cup V((\Theta^2_{x,y_j})^*) \cup V(P_i)]$ with root $x$, which is a contradiction.

	Thus, for every $m\in\poi_{\gamma}$, there exists a vertex $p_{i_m}\in P_{i_m}$ such that $p_{i_m}$ has neighbors in at least $t$ paths of $(\Theta^2_{x,y_j})^*$.
	Recall that $G$ is $t$-tidy.
	By Theorem \ref{thm:ramsey}, there exists a set of $\theta$ vertices $p_{j_1},\dots,p_{j_{\theta}}$ that form a stable set.
	Since $\theta = {q' \choose t}t$, there exist $t$ vertices $p_{k_1},\dots,p_{k_t}$ that all have neighbors in the same $t$ paths $Q_1,\dots,Q_t$ of $(\Theta^2_{x,y_j})^*$.
	However, contracting the paths $Q_1,\dots,Q_t$ yields an induced $K_{t,t}$-minor, which is a contradiction as $G$ is $t$-tidy.
	This completes the proof of \eqref{sta:cleanup_complete}.

	\

	By (\ref{sta:cleanup_complete}) and Theorem \ref{thm:ramsey}, $D^-$ contains a stable set $S$ of size $k$.
	This implies that there exists a set of $k$ $(x,z)$-tailed thetas $\{(\Theta^2_{x,y_i}, P_i) \mid i\in S\}$ such that, for all distinct $i,j\in S$, the path $P_i$ has neighbors in at most $s$ paths of $(\Theta^2_{x,y_j})^*$.
	For every pair $P_i, \Theta^2_{x,y_j}$ where $i\neq j$ remove from $(\Theta^2_{x,y_j})^*$ the paths that have a neighbor in $P_i$.
	This will remove at most $ks$ paths from each $\Theta^2_{x,y_j}$ resulting in a collection of $q$-thetas $\Theta^3_{x,y_{j_1}},\dots,\Theta^3_{x,y_{j_k}}$ and their corresponding tails $P_{j_1},\dots,P_{j_k}$ that together form a clean $(k,q)$-kite as claimed.
	This concludes the proof.
	
\end{proof}

\section{From kites to a dome}
\label{sec:general_step}

We define a \textit{binary tree} to be a rooted tree $T$ such that every node $t\in V(T)$ has at most two children.
A {\em full binary tree of radius $r$} is a binary tree in which every vertex that is not a leaf has exactly two children, and every leaf is at distance exactly $r$ from the root.
Let $T_r$ denote the full binary tree of radius $r$.
For $i\in\poi_r$, we define the {\em $i$-th level set of} $T_r$, denoted by $L_i(T_r)$, as the set of vertices at distance $i$ from the root of $T_r$.
Note that $|L_i(T_r)| = 2^i$ and that $L_r(T_r)$ is the set of leaves of $T_r$.
For a subdivision $T_r'$ of $T_r$, we define $L_i(T'_r)$ to be $L_i(T_r)$.
An {\em $r$-fleet} in $G$ is a pair $(x,T'_r)$, where $T'_r$ is an induced subgraph of $G$ isomorphic to a subdivision of $T_r$, $x\in V(G)\setminus V(T'_r)$, and $N(x)\cap V(T'_r)=L_r(T'_r)$.
We call $x$ the {\em anchor} of the fleet.
As with other structures, we will refer to a fleet and the graph induced by its vertex set interchangeably.

\begin{figure}
	\centering
	\includegraphics[width=0.25\linewidth]{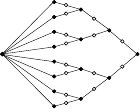}
	\caption{A $3$-fleet.}
	\label{fig:fleet}
\end{figure}

In this section, we complete the proof of Theorem \ref{thm:banana_has_cycles}.
The proof strategy is to repeatedly apply Theorem \ref{thm:banana_to_clean_kite} to show that $G$ contains either an $(s,g_{\blackref{thm:banana_has_cycles}})$-dome or an $s$-fleet $F$.
However, as the next lemma shows, $F$ also contains an $(s,1)$-dome.
This completes the proof.

\begin{lemma}
\label{lem:fleet}

	Let $F$ be an $s$-fleet in a graph $G$ with anchor $x$.
	Then $G$ contains an $(s,1)$-dome with root $x$.
	
\end{lemma}

\begin{proof}

	Let $T'_s$ be the subdivided binary tree with root $r$ in $F$.
	Let $P$ be a path from $r$ to a vertex $v$ in $L_s(T'_s)$, $L=N(P)\cap T'_s$, and $P'=P\setminus\{v\}$.
	Note that $|L|= s$, every vertex in $L$ has exactly one neighbor in $P'$, and every vertex of $P'$ is adjacent to at most one vertex of $L$.
	Finally, observe that $G[V(F)\setminus V(P')]$ contains a subdivided star $S$ with root $x$ and leaf set $L$ such that $V(S)\setminus L$ is anticomplete to $V(P')$.
	Thus, $(S,P')$ is an $(s,1)$-dome with root $x$ in $G$ as required.
	
\end{proof}

\bananaHasCycles*

\begin{proof}

	Define
	\allowdisplaybreaks
	\begin{align*} 
		k_1 &= f_{\blackref{thm:banana_to_clean_kite}}(1,l,2,s,t)\\
		k_1' &= l^{k_1-1}\\
		k_i &= f_{\blackref{thm:banana_to_clean_kite}}(k_{i-1}',l,2,s,t)\ \text{ for } i=2,\dots,s\\
		k_i' &= l^{k_i-1}\ \text{ for } i=2,\dots,s\\
		f_{\blackref{thm:banana_has_cycles}} &= k_s'\\
		g_{\blackref{thm:banana_has_cycles}} &= g_{\blackref{thm:banana_to_clean_kite}}(t)
	\end{align*}

	Suppose that $G$ does not contain an $(s, g_{\blackref{thm:banana_has_cycles}})$-dome with root $x$; we will show that $G$ contains an $s$-fleet.

	We start with the following.
	
	\sta{\label{rigidbanana}
	Let $l,q \in \poi$ and let $B_{x,z}=(\{x, z\}, \mca{P})$ be an $l$-dense banana with $|\mca{P}| \geq l^{q-1}$.
	Then $B_{x,z}$ contains a rigid $l$-dense $q$-banana $B_{x,z}'$.}

	Define a graph $H$ with vertex set $\mca P$ and an edge between paths $P,Q\in \mca P$ if and only if $P^*$ is not anticomplete to $Q^*$.
	Since $B_{x,z}$ is $l$-dense, $H$ does not contain a stable set of size $l$.
	Therefore, Theorem \ref{thm:ramsey} implies that $H$ contains a clique $\mca{P}'$ of size $q$.
	Now, $B'_{x,z}=(\{x,z\}, \mca{P}')$ is a rigid banana, as required.
	This proves \eqref{rigidbanana}.

	\

	Next, we prove:

	\sta{\label{sta:has_a_fleet}
	Let $i\in \poi_s$.
	Let $B_{x,z}$ be an $l$-dense $k_i'$-banana that is in a $t$-tidy graph with no $(s,g_{\blackref{thm:banana_has_cycles}})$-dome with root $x$.
	Then $B_{x,z}$ contains an $i$-fleet with anchor $x$.}

	We prove \eqref{sta:has_a_fleet} by induction on $i$.
	We start with the base case.
	Let $B_{x,z} = (\{x,z\},\mca P)$ be an $l$-dense $t$-tidy $k_1'$-banana with no $(s,g_{\blackref{thm:banana_has_cycles}})$-dome with root $x$.
	By \eqref{rigidbanana}, $B_{x,z}$ contains a rigid $k_1$-banana $B'_{x,z}$.
	By Theorem \ref{thm:banana_to_clean_kite}, $B_{x,z}'$ contains a clean $(x,z)$-tailed $(1,2)$-kite $\mca K$.
	Let $\Theta_{x,y}$ be the only theta of $\mca K$ (in fact, $\Theta_{x,y}$ is an induced cycle).
	Then $\Theta_{x,y}$ is a $1$-fleet with anchor $x$.
	This completes the base case.

	Now, suppose that \eqref{sta:has_a_fleet} holds for all $i' < i$, and let $B_{x,z}$ be an $l$-dense $t$-tidy $k_i'$-banana with no $(s,g_{\blackref{thm:banana_has_cycles}})$-dome with root $x$.
	By \eqref{rigidbanana}, $B_{x,z}$ contains a rigid $l$-dense $k_i$-banana $B_{x,z}'$.
	By Theorem \ref{thm:banana_to_clean_kite}, $B_{x,z}'$ contains a clean $(x,z)$-tailed $(k_{i-1}',2)$-kite $\mca K$.
	Let $B''_{x,z}$ be obtained from $\mca K$ by replacing every theta $\Theta_{x,y_j}$ of $\mca{K}$ by an edge $xy_j$.

	Since $B''_{x,z}$ is an induced minor of $B'_{x,z}$ obtained from $\mca{K}$ by contracting only edges with at least one end of degree two in $\mca{K}$, and since $\{y_1,\dots,y_{k_{i-1}'}\}$ is a stable set, it follows that $B''_{x,z}$ is $t$-tidy, $l$-dense and does not contain an $(s,g_{\blackref{thm:banana_has_cycles}})$-dome with root $x$.

	Inductively, $B''_{x,z}$ contains an $(i-1)$-fleet $F'$ with anchor $x$.
	Renumbering if necessary, we may assume that $y_1, \dots, y_{2^{i-1}}$ are the neighbors of $x$ in $F'$.
	Replacing each of the edges $xy_j$ in $F'$ by the cycle $\Theta_{x,y_j}$, we obtain an $i$-fleet $F$ with anchor $x$, as claimed.
	This completes the proof of \eqref{sta:has_a_fleet}.

	\

	Now, let $B_{x,z}$ be as in the statement of the theorem.
	By \eqref{sta:has_a_fleet}, $B_{x,z}$ contains an $s$-fleet $F$ with anchor $x$.
	By Lemma \ref{lem:fleet}, $F$ contains an $(s,1)$-dome with root $x$, which is a contradiction.
	This concludes the proof.

\end{proof}

\section{Detecting induced cycles}
\label{sec:algo}

In this section, we prove Theorem \ref{thm:algo_bounded}.

Throughout this and the next section, we will be using the notation $G=(V,E)$ with $|V|=n$ for $n\in \poi$.
For a tree decomposition $(T,\chi)$ where $T$ is a binary tree, and a node $t\in V(T)$ whose parent is $p\in V(T)$, we define the \textit{adhesion} of $t$, denoted $A(t)$, to be $\chi(t)\cap \chi(p)$.
The adhesion of the root node $r$ is the empty set.

Let $t\in V(T)$ be a node of $T$ and let $T_t$ be the subtree of $T$ rooted at $t$.
We call the nodes of $V(T_t)$ the {\em descendants of $t$} in $T$.
Additionally, let $G_t$ be the graph $G[\cup_{v\in V(T_t)} \chi(v)]$.
We will need the following theorem:

\begin{theorem}[Bodlaender, Kloks \cite{Bodlaender,Kloks}]
\label{thm:tree_decomp}

	Let $k \in \poi$ be a constant.
	For a graph $G$, there exists an algorithm with running time $O(n)$ that decides whether $\tw(G)\le k$ and if so outputs a tree decomposition $(T,\chi)$ of $G$ of width at most $k$ such that $T$ is a binary tree with $O(n)$ nodes.

\end{theorem}

Let $t\in T$ be a node and let $Q\subseteq \chi(t)$ be an ordered set $\{q_1,\dots,q_l\}$ such that\break $N(q_i)\cap Q \subseteq \{q_{i-1},q_{i+1}\}$ for $i\in \poi_l$, where subscripts are modulo $l$ throughout this proof.
Let $R\subseteq \{(q_1,q_2),\dots,(q_{l-1},q_l),(q_l,q_1)\}$ such that if $q_iq_{i+1}\in E(G)$, then $(q_i,q_{i+1}) \in R$.
Let $X\subseteq A(t)\setminus Q$.
We define a {\em path system $S_{Q,R,X}$ in $G_t$} to be a set of paths $\{P_r \mid r\ \in R\}$ in $G_t$ such that for every $r\in R$,
\begin{itemize}
	\item $P_r$ has ends $q_i,q_{i+1}$ where $r=(q_i,q_{i+1})$,
	\item $P_r^*$ is disjoint from $P_{r'}^*$ for $r' \neq r$,
	\item $P_r^* \subseteq V(G_t) \setminus (\chi(t)\cup N(X))$, and
	\item every vertex of $P_r^*$ has degree two in $G[V(S_{Q,R,X})\cup Q]$.
\end{itemize}
We will treat $S_{Q,R,X}$ and $G[V(S_{Q,R,X})]$ interchangeably.
Note that a path system is either a cycle or a graph in which every component is a path.
The \textit{length of} a path system $S_{Q,R,X}$ is defined as $|E(S_{Q,R,X})|$.
The \textit{length variation of} $Q,R,X$ in $G_t$, denoted by $\lv(Q,R,X,G_t)$, is the set of all possible lengths of a path system $S_{Q,R,X}$ in $G_t$, i.e., \[\lv(Q,R,X,G_t)=\{l\in \poi \mid \text{there exists a path system } S_{Q,R,X} \text{ in } G_t \text{ of length } l \}.\]
We are now ready to present the first algorithm.

\algoBounded*

\begin{proof}

	First, we describe an algorithm that outputs a set $\mca C$ of size $\min \{c,\cl(G)\}$ such that for every $l\in \mca C$, $G$ contains an induced cycle of length $l$.
	Later, we will describe how to add backtracking to this algorithm to output not only the lengths but the actual induced cycles.
	
	Using the algorithm from Theorem \ref{thm:tree_decomp}, find a tree decomposition $(T,\chi)$ of width at most $k$.
	Assume that every node $t\in V(T)$ that is not a leaf has two children where if needed let $\chi(t')=\varnothing$ for the missing child $t'$.
	This increases the number of nodes by at most a factor of two, so the total number of nodes is still $O(n)$.
	
	We will solve the problem via dynamic programming over $T$.
	We will maintain a set $\mca C$ (initialized to $\varnothing$) that contains lengths of induced cycles in $G$.
	Once the size of $\mca C$ reaches $c$, we simply output it.
	Additionally, for every node $t\in V(T)$, we will maintain a table $M_t$ indexed by $Q,R,X$ such that each entry $M_t[Q,R,X]$ is a set of integers with the property that if $l\in M_t[Q,R,X]$, then there exists a path system $S_{Q,R,X}$ in $G_t$ of length $l$, and moreover $|M_t[Q,R,X]|\le c$.
	By the end of the run of the algorithm, $M_t[Q,R,X]$ will contain a subset of $\lv(Q,R,X,G_t)$ of size $\min\{c,|\lv(Q,R,X,G_t)|\}$.
		
	Let $t\in V(T)$ be a node that is not a leaf and let $t_1,t_2$ be the two children of $t$.
	Intuitively, our algorithm does the following: to compute an entry $M_t[Q,R,X]$, it guesses which vertices of $\chi(t_1)\setminus A(t_1)$ and $\chi(t_2)\setminus A(t_2)$ are needed to connect the pairs in $R$ and in what order they appear on the corresponding paths, and then queries appropriate path systems in $M_{t_1}$ and $M_{t_2}$.
	While doing the lookups in $M_{t_1}$ and $M_{t_2}$, we specify the sets $X_1\subseteq A(t_1),X_2\subseteq A(t_2)$ of vertices that need to be anticomplete to the interiors of the paths of the path systems found in $G_{t_1},G_{t_2}$ to ensure correctness.
	We then also check if $S_{Q,R,X}$ forms an induced cycle in $G_t$ and record its length.

	We process the nodes bottom up, from the leaves to the root, processing each node once all of its children have been processed.
	After all the nodes have been processed, we either output ``$\cl(G)<c$'' if the size of $\mca C$ is smaller than $c$, or otherwise we output $\mca C$.
	Let us now describe how to process each node.
	We will prove correctness later.
	
	\
	
	\noindent
	\textit{Base Case:}
	
	\
	
	We start by processing the leaves of $T$.
	Let $t\in V(T)$ be a leaf node in $T$.
	For every ordered set $Q = \{q_1,\dots,q_l\}\subseteq \chi(t)$ such that $N(q_i)\cap Q \subseteq \{q_{i-1},q_{i+1}\}$ for $i\in \poi_l$, let $R = \{(q_i,q_{i+1})\mid q_iq_{i+1}\in E(G)\}$.
	For all sets $X\subseteq A(t)\setminus Q$, we add $|E(Q)|$ to $M_t[Q,R,X]$.
	We then check if $Q$ is an induced cycle in $G$, and if so, add $|E(Q)|$ to $\mca C$.
	
	\
	
	\noindent
	\textit{General Case:}
	
	\

	Let $t\in V(T)$ be a node such that all the descendants of $t$ (except for $t$) have already been processed.
	Denote the two children of $t$ by $t_1$ and $t_2$.
	Our goal is to compute the table $M_t$ and then query it to record the lengths of all induced cycles in $G_t$ that have not yet been recorded while processing $G_{t_1}$ and $G_{t_2}$.
	
	Iterate over 
	\begin{itemize}
		\item all ordered sets $Q = \{q_1,\dots,q_l\}\subseteq \chi(t)$ such that $N(q_i)\cap Q \subseteq \{q_{i-1},q_{i+1}\}$ for $i\in \poi_l$,
		\item all sets $R\subseteq \{(q_1,q_2),\dots,(q_{l-1},q_l),(q_l,q_1)\}$ such that if $q_iq_{i+1}\in E(G)$, then $(q_i,q_{i+1}) \in R$, and
		\item all sets $X\subseteq A(t)\setminus Q$.
	\end{itemize}
	
	For every such triple $Q,R,X$, let $R' = R \setminus \{(q_i,q_{i+1})\mid q_iq_{i+1}\in E(G)\} = \{r_1,\dots,r_{|R'|}\}$, i.e., $R'$ contains the pairs of $R$ that are not yet connected.
	If $R'=\emptyset$, add $|E(Q)|$ to $M_t[Q,R,X]$.
	Check if $Q$ is an induced cycle in $G$ and if so, add $|E(Q)|$ to $\mca C$.
	If $R'\neq\emptyset$, do the following.
	Iterate through all vectors $I$ in $\{1,2\}^{|R'|}$ such that if $I_j=i$, then $\{q,q'\}\subseteq A(t_i)$ where $r_j=(q,q')$.
	Intuitively, the vector $I$ indicates whether a pair $r\in R'$ should be connected by a path in $G_{t_1}$ or $G_{t_2}$.
	For every such vector, iterate through all $|R'|$-tuples $(V_1,\dots,V_{|R'|})$ such that
	\begin{itemize}
		\item every $V_j$ is an ordered subset of $\chi(t_i)\setminus (A(t_i) \cup N(X))$ where $i = I_j$,
		\item $V_i$ and $V_j$ are disjoint and anticomplete for $i\neq j$, and
		\item every $V_j$ is anticomplete to $Q\setminus \{q,q'\}$ where $r_j=(q,q')$.
	\end{itemize}
	Every set $V_j$ contains the vertices of $\chi(t_{I_j})$ used to connect the pair $r_j$.
	
	For every pair $r_j=(q,q')\in R'$ such that $I_j=i$ define $Z_i(q)$ to be a list of elements of $V_j$ in the same order as they appear in $V_j$, and otherwise define $Z_i(q)$ to be an empty list.
	For $i\in \{1,2\}$, define
	\begin{itemize}
		\item $R_i'=\{r_j \in R'\mid I_j = i\}$ (these are the pairs of $R'$ assigned to be connected through $G_{t_i}$),\\
		\item $Q_i' = \{q\in Q\mid (q,q')\in R_i' \text{ or } (q',q)\in R_i' \text{ for some } q' \in Q\} = \{q_{p_1},\dots,q_{p_{k_i}}\}$ where $p_1<p_2<\dots<p_{k_i}$ (these are the vertices of $R_i'$ in the same order as they appear in $Q$),\\
		\item $Q_i = \{q_{p_1},Z_i(q_{p_1}),q_{p_2},Z_i(q_{p_2}),\dots,q_{p_{k_i}},Z_i(q_{p_{k_i}})\} = \{q_1^i,\dots,q_{m_i}^i\}$ to be an ordered set, \\
		\item $R_i=\{(q_1^i,q_2^i),\dots,(q_{m_i-1}^i,q_{m_i}^i),(q_{m_i}^i,q_1^i)\}\setminus \{(q,q')\mid q,q'\in Q_i'\text{ and } qq' \notin E(G) \text{ and } \\ (q,q')\notin R_i'\}$,\\
		\item $X_i = (A(t_i)\cap (Q\cup X))\setminus Q_i$.
	\end{itemize}
	
	For $i\in\{1,2\}$, if $R_i=\varnothing$, set $L_i=\{0\}$.
	Otherwise, let $L_i=M_{t_i}[Q_i,R_i,X_i]$.
	Let $s=|E(Q) \setminus (E(Q_1) \cup E(Q_2))|$.
	For every pair $l_1\in L_1,l_2\in L_2$, add $s+l_1+l_2$ to $M_t[Q,R,X]$.
	If $R=\{(q_1,q_2),\dots,(q_{l-1},q_l),(q_l,q_1)\}$ and $l\ge 2$, then also add $s+l_1+l_2$ to $\mca C$.
	This completes the description of the algorithm.
	
	\
	 
	\noindent
	\textit{Running time:}
	
	\
	
	For every bag, the number of sets $Q,R,X$, vectors $I$, and $|R'|$-tuples $(V_1,\dots,V_{|R'|})$ is bounded by $O(1)$ as the size of each bag is at most $k+1$ (which is a constant).
	Forming the sets $R_1,R_2,Q_1,Q_2,X_1,X_2$ and checking that the set $Q$ and the tuple $(V_1,\dots,V_{|R'|})$ are valid can be done in $O(k^2)=O(1)$ by looping over an appropriate submatrix of the adjacency matrix of $G$.
	Since every queried set in the tables $M_{t_1},M_{t_2}$ contains at most $c$ elements, looping through it can also be done in $O(1)$.
	Thus, there exists a function $f:\poi^2 \to \poi$ such that the running time per node is $O(f(k,c))=O(1)$.
	The running time of the algorithm from Theorem \ref{thm:tree_decomp} is $O(n)$, and since it returns a tree decomposition with $O(n)$ nodes, the total running time of our algorithm is $O(n)$.
	
	\
	
	\noindent
	\textit{Proof of correctness:}
	
	\
	
	Consider an unrestricted version of our algorithm so as to not restrict the size of each entry in the tables $M_t$ and the set $\mca C$ to at most $c$.
	If the unrestricted version outputs a correct list of all induced cycle lengths of $G$, then to decide whether $\cl(G)\ge c$, it is enough to detect only $c$ path systems of different lengths for every triple $Q,R,X$ as all such path systems are interchangeable.
	Thus, for the rest of the proof, we may work with the unrestricted version of our algorithm.

	We prove that the algorithm produces a correct output by combining claims \eqref{sta:table_is_correct}, \eqref{sta:table_is_correct2}, \eqref{sta:each_cycle_length_is_considered}, and \eqref{sta:lengths_are_valid} presented below.
	We will prove claims \eqref{sta:table_is_correct} and \eqref{sta:table_is_correct2} by induction on the size of the tree $T$.
	
	\sta{\label{sta:table_is_correct} Let $t\in V(T)$, $Q=\{q_1,\dots,q_l\}$, $R\subseteq \{(q_1,q_2),\dots,(q_{l-1},q_l), (q_l,q_1)\}$, and $X\subseteq A(t)\setminus V(Q)$.
	If $G_t$ contains a path system $S_{Q,R,X}$, then $M_t[Q,R,X]$ contains the length of $S_{Q,R,X}$.}
	
	{\em Base case:}
	Note that in the base case when $T$ is just one node, $R = \{(q_i,q_{i+1})\mid q_iq_{i+1}\in E(G)\}$.
	Thus, the length of $S_{Q,R,X}$ is $|E(Q)|$, which will be recorded by the algorithm.
	
	\
	
	{\em Inductive step:}
	At some point, the algorithm will consider the triple $Q,R,X$ as it iterates over all such triples.
	Since $S_{Q,R,X}$ is a path system, the set $Q$ fulfills the condition that $N(q_i)\cap Q \subseteq \{q_{i-1},q_{i+1}\}$ for $i\in\poi_l$, and the set $R$ fulfills the condition that if $q_iq_{i+1}\in E(G)$, then $(q_i,q_{i+1}) \in R$.
	Thus, these checks will succeed.	Let $R'$ be defined as in the algorithm description.
	If $R'=\emptyset$, then our algorithm will directly record $|E(Q)|$.
	Otherwise, by the definition of a path system, it follows that for every $r_i=(q,q')\in R'$, the vertices $q$ and $q'$ are connected by a path $P_i$ in either $G_{t_1}$ or $G_{t_2}$.
	
	Eventually, our algorithm will consider a vector $I$ and an $|R'|$-tuple $(V_1,\dots,V_{|R'|})$ such that $V_j=P_j\cap (\chi(t_i) \setminus (A(t_i)\cup N(X)))$ where $I_j=i$ for every $j\in \poi_{|R'|}$.
	It is also clear that
	\begin{itemize}
		\item $V_i$ and $V_j$ are disjoint and anticomplete for $i\neq j$, and
		\item that every $V_j$ is anticomplete to $Q\setminus \{q,q'\}$ where $r_j=(q,q')$.
	\end{itemize}
	
	Now, observe that the path systems found for $Q_1,R_1,X_1$ and $Q_2,R_2,X_2$ together with the set $Q$ form a path system for $Q,R,X$.
	Inductively, the entries $M_{t_1}[Q_1,R_1,X_1]$ and $M_{t_2}[Q_2,R_2,X_2]$ have been computed correctly.
	Therefore, they contain the lengths $l_1,l_2$ of path systems $S_{Q_1,R_1,X_1},S_{Q_2,R_2,X_2}$ that together with $Q$ form a path system of the same length as $S_{Q,R,X}$.
	This completes the proof of \eqref{sta:table_is_correct}.
	
	\sta{\label{sta:table_is_correct2} Let $t\in V(T)$, $Q=\{q_1,\dots,q_l\}$, $R\subseteq \{(q_1,q_2),\dots,(q_{l-1},q_l), (q_l,q_1)\}$, and $X\subseteq A(t)\setminus V(Q)$.
	If $m\in M_t[Q,R,X]$, then there exists a path system $S_{Q,R,X}$ in $G_t$ of length $m$.}
	
	{\em Base case:} In the base case when $T$ is just one node, when $m$ is added to $M_t[Q,R,X]$ it holds that $m=|E(Q)|$ and that $Q$ forms a path system.
	Thus, the claim holds.
	
	\
	
	{\em Inductive step:}
	Consider the first time that $m$ is added to $M_t[Q,R,X]$.
	Let $R',Q_1,Q_2,R_1,\allowbreak R_2,X_1,X_2$ and $l_1,l_2$ be defined as in the algorithm description.
	If $R'=\emptyset$, then $Q$ already forms a valid path system.
	Otherwise, by induction, the entries $M_{t_1}[Q_1,R_1,\allowbreak X_1]$ and $\break M_{t_2}[Q_2,R_2,X_2]$ have been computed correctly.
	Therefore, there exist path systems $S_{Q_1,R_1,X_1}, \break S_{Q_2,R_2,X_2}$ in $G_{t_1},G_{t_2}$ of lengths $l_1,l_2$ respectively.
	
	By the definition of the sets $Q_1,Q_2,R_1,R_2,\allowbreak X_1,X_2$, we have that for every pair $(q,q')\in R$, either $qq'\in E(G)$ or $q$ and $q'$ are connected by a path in one of $S_{Q_1,R_1,X_1}, S_{Q_2,R_2,X_2}$.
	By the conditions imposed on the $|R'|$-tuple $(V_1,\dots,V_{|R'|})$, it follows that for every $i\in \poi_{|R'|}$, every vertex of $V_i$ has degree two in $Q \cup S_{Q_1,R_1,X_1} \cup S_{Q_2,R_2,X_2}$ and is anticomplete to $Q\setminus \{q,q'\}$ where $r_i=(q,q')$.
	Therefore, $Q \cup S_{Q_1,R_1,X_1} \cup S_{Q_2,R_2,X_2}$ forms a path system $S_{Q,R,X}$ in $G_t$ of length $m$ as required.
	This completes the proof of \eqref{sta:table_is_correct2}.

	\pagebreak[2]
	
	\sta{\label{sta:each_cycle_length_is_considered} If $G$ contains an induced cycle $C$, then $\mca C$ contains $|C|$.}
	
	Recall that we process nodes of $T$ bottom up.
	Consider the first node $t\in V(T)$ such that $C$ is contained in $G_t$.
	Note that $|C\cap \chi(t)|\ge 2$ as otherwise $C$ is fully contained in either $G_{t_1}$ or $G_{t_2}$, both of which have been processed before $t$.
	Suppose that $Q=C\cap \chi(t) = \{q_1,\dots,q_l\}$ such that $q_1,\dots,q_l$ appear in $C$ in this order.
	Note that by the definition of a tree decomposition, every pair $q_i,q_{i+1}$ is connected by either an edge or a path in $G_{t_1}$ or $G_{t_2}$.
	Then $C$ forms a path system $S_{Q,R,X}$ with $R=\{(q_1,q_2),\dots,(q_{l-1},q_l), (q_l,q_1)\}$ and $X=\varnothing$.
	By \eqref{sta:table_is_correct}, $M_t[Q,R,X]$ is computed correctly, and thus it contains $|C|$.
	This completes the proof of \eqref{sta:each_cycle_length_is_considered}.
	
	\sta{\label{sta:lengths_are_valid} If $\mca C$ contains $l\in \poi$, then $G$ contains an induced cycle $C$ of length $l$.}
	
	Consider the first time $l$ has been added to $\mca C$.
	At that point, the algorithm has just added $l$ to $M_t[Q,R,X]$ for some node $t$, set $Q=\{q_1,\dots,q_s\}$, set $X \subseteq A(t)\setminus Q$, and set $R = \{(q_1,q_2),\dots,(q_{s-1},q_s), (q_s,q_1)\}$.
	Note that the corresponding path system $S_{Q,R,X}$ of length $l$ is indeed an induced cycle $C$ of length $l$ with $C\cap \chi(t) = Q$.
	This completes the proof of \eqref{sta:lengths_are_valid}.

	\

	\noindent
	\textit{Backtracking:}

	\

	First, we extend the tables $M_t$ so that for every $l\in M_t[Q,R,X]$, when we add $l$ to $M_t[Q,R,X]$, we additionally store the current sets $Q_1,Q_2,R_1,R_2,X_1,X_2$ and the lengths $l_1,l_2$ queried in the tables $M_{t_1}[Q_1,R_1,X_1],M_{t_2}[Q_2,R_2,X_2]$ that were used to obtain a path system $S_{Q,R,X}$ of length $l$.
	This is only $O(1)$ overhead per entry.
	
	Then, when we detect an induced cycle of a new length, we recursively look up all the vertices used to form this induced cycle.
	This requires $O(1)$ time per node, and as there are $O(n)$ nodes, the whole lookup takes time $O(n)$ per induced cycle.
	Since we only record at most $c$ induced cycle lengths, the overall time spent for backtracking is thus also only $O(n)$.
	
\end{proof}

\section{Finding three induced cycles}
\label{sec:three_cycles}

In this section, we describe how to apply the algorithm from Theorem \ref{thm:algo_bounded} to decide whether $\cl(G)\ge 3$ in polynomial time.
Before we provide this algorithm, we need to introduce a result from \cite{non-shortest-faster}.

\begin{theorem}[Chiu, Lu \cite{non-shortest-faster}]
\label{thm:non_shortest_path}

	Let $G$ be a graph and let $u,v\in V(G)$ be two vertices.\ There exists an algorithm with running time $O(n^{4.75})$ that decides whether there exists a $u\dd v$ path in $G$ that is not a shortest $u\dd v$ path.
	If such a path exists, the algorithm will output it.

\end{theorem}

Now, we are ready to present our second algorithm.

\algoThree*

\begin{proof}
	
	Suppose that $G$ is $\{K_3,K_{2,2}\}$-free.
	Let $k=f_{\blackref{thm:main_result}}(3,2)$ be the threshold from Theorem \ref{thm:main_result}.
	Test whether $\tw(G) \le k-1$ using Theorem \ref{thm:tree_decomp}.
	If so, then apply Theorem \ref{thm:algo_bounded} and output the result.
	Otherwise, remove arbitrary vertices from $G$ one by one until the treewidth of the resulting graph $G'$ is at most $k-1$.
	Add the last removed vertex back to $G'$ to obtain $G''$.
	Observe that $\tw(G'')= k$ and that by Theorem \ref{thm:main_result}, $G''$ (and hence $G$) contains three cycles of distinct lengths.
	Apply Theorem \ref{thm:algo_bounded} to $G''$ and output the result.
	The correctness follows since $G''$ is an induced subgraph of $G$.
	The running time of the algorithms from Theorems \ref{thm:tree_decomp} and \ref{thm:algo_bounded} is $O(n)$ and thus the total running time of this step is $O(n^2)$.
	
	Now, suppose that $G$ contains a $K_3$.
	Thus, the shortest induced cycle in $G$ has length three.
	We can find one cycle $C^0$ of length three by iterating over all triples of $V(G)$ in time $O(n^3)$.
	Then, for every edge $uv\in E(G)$, do the following:
	\begin{enumerate}
		\item Construct a graph $G_{uv}$ from $G$ by removing the edge $uv$ and vertices in $N(u) \cap N(v)$.
		\item Find a shortest $u\dd v$ path $P$ in $G_{uv}$, if one exists, by running breadth first search.
	If $P$ exists, let $C^1_{uv}$ be the induced cycle of $G$ with vertex set $P$.
		\item Find a non-shortest $u\dd v$ path $N$ in $G_{uv}$, if one exists, using Theorem \ref{thm:non_shortest_path}.
	If $N$ exists, let $C_{uv}^2$ be the induced cycle of $G$ with vertex set $N$.
		\item If $P$ above does not exist, let $L_{uv}=\varnothing$.
	If only $P$ above exists, let $L_{uv}=\{C^1_{uv}\}$.
	If both $P$ and $N$ above exist, let $L_{uv}=\{C_{uv}^1,C_{uv}^2\}$.
	\end{enumerate}
	
	By Theorem \ref{thm:non_shortest_path}, the set $L_{uv}$ can be constructed in time $O(n^{4.75})$ for every edge $uv$.
	Construct a set $L$ of cycles by starting with $L=\varnothing$, examining every $L_{uv}$ in turn, and adding the cycle $C \in L_{uv}$ to $L$ if and only if $L$ does not yet contain a cycle of length $|C|$.
	This check can be done in time $O(n)$ for every edge $uv$.
	Thus, the total running time is $O(n^{6.75})$.

	If $|L|<2$, output ``$\cl(G) <3$''.
	If $|L| \geq 2$, choose two elements $C^1,C^2 \in L$, and output ``$\cl(G)\ge 3$'' and the list $C^0, C^1, C^2$.

	We now prove correctness in the case when $G$ contains a $K_3$.
	Suppose first that the algorithm outputs ``$\cl(G) \geq 3$''.
	Since $N_{G_{uv}}(u) \cap N_{G_{uv}}(v)=\varnothing$ for every $uv \in E(G)$, it follows that $|C|>3$ for every $C \in L$.
	Thus, the lengths of the cycles $C^0,C^1,C^2$ are all distinct, and so the algorithm performs correctly.

	Now, suppose that the algorithm outputs ``$\cl(G)<3$''.
	It follows that $|L|<2$.
	Assume for a contradiction that $\cl(G) \geq 3$.
	Then there exist distinct $l_1, l_2 >3$ and induced cycles $C_1$, $C_2$ in $G$ such that $|C_i|=l_i$.
	For $i \in \poi_2$, let $u_iv_i$ be an edge of $C_i$, and let $P_i$ be the path obtained from $C_i$ by removing the edge $u_iv_i$.
	Then $P_i$ is a path from $u_i$ to $v_i$ in $G_{u_iv_i}$.
	If $P_i$ is not a shortest path from $u_i$ to $v_i$ in $G_{u_iv_i}$, then $|L_{u_iv_i}|=2$, contrary to the fact that $|L| < 2$.
	Thus, for $i \in \poi_2$, $P_i$ is a shortest path from $u_i$ to $v_i$ in $G_{u_iv_i}$.
	Therefore, $|C^1_{u_1v_1}|=|C_1|$ and $|C^1_{u_2v_2}|=|C_2|$, and thus, $L$ contains cycles of lengths $l_1$ and $l_2$, again contrary to the fact that $|L|<2$.
	This proves the algorithm works correctly when $G$ contains a $K_3$.

	Therefore, we may assume that $G$ is $K_3$-free but contains a $K_{2,2}$.
	Thus, the shortest induced cycle in $G$ has length four.
	We can find one cycle $C^0$ of length four by iterating over all 4-sets of $V(G)$ in time $O(n^4)$.
	Then, for every path $Q=v_1 \dd v_2 \dd v_3$ in $G$, do the following:
	\begin{enumerate}

		\item Construct a graph $G_Q$ from $G$ by removing $N(v_2)\setminus \{v_1,v_3\}$ and $N(v_1)\cap N(v_3)$.
		\item Find a shortest $v_1\dd v_3$ path $P$ in $G_Q$, if one exists, by running breadth first search.
	If $P$ exists, let $C^1_Q$ be the induced cycle of $G$ with vertex set $P\cup \{v_2\}$.
		\item Find a non-shortest induced $v_1\dd v_3$ path $N$ in $G_Q$, if one exists, using Theorem \ref{thm:non_shortest_path}.
	If $N$ exists, let $C_Q^2$ be the induced cycle of $G$ with vertex set $N\cup \{v_2\}$.
		\item If $P$ above does not exist, let $L_Q=\varnothing$.
	If only $P$ above exists, let $L_Q=\{C^1_Q\}$.
	If both $P$ and $N$ above exist, let $L_Q=\{C_Q^1,C_Q^2\}$.
	\end{enumerate}
	
	By Theorem \ref{thm:non_shortest_path}, the set $L_Q$ can be constructed in time $O(n^{4.75})$ for every path $Q=v_1\dd v_2\dd v_3$.
	Construct a set $L$ of cycles by starting with $L=\varnothing$, examining every $L_Q$ in turn, and adding the cycle $C \in L_Q$ to $L$ if and only if $L$ does not yet contain a cycle of length $|C|$.
	This check can be done in time $O(n)$ for every path $Q$.
	Thus, the total running time is $O(n^{7.75})$.
	
	If $|L|<2$, output ``$\cl(G) <3$''.
	If $|L| \geq 2$, choose two elements $C^1,C^2 \in L$, and output ``$\cl(G)\ge 3$'' and the list $C^0, C^1, C^2$.
	
	Observe that deleting $N(v_2)\setminus \{v_1,v_3\}$ and $N(v_1)\cap N(v_3)$ ensures that every cycle $C\in L$ has length at least five since $Q$ has length two, and $P$ and $N$ have length at least three.
	Then the correctness follows by the same argument as in the case of $K_3$.
	This concludes the proof.
	
\end{proof}

\section{Acknowledgments}

We thank Paul Seymour and Nicolas Trotignon for suggesting the problem to us and for several helpful discussions.

\section{AI Disclosure}

We used GPT-6.1 Sol to proofread the manuscript.
No AI tools were used to create the proof.

\bibliographystyle{plain}
\bibliography{ref}

\end{document}